%% file: essspec_hilb_comp.tex
\pdfoutput=1

\documentclass[reqno,11pt]{amsart}
\usepackage{geometry}
\usepackage[T1]{fontenc}
\usepackage{type1ec}
\usepackage[utf8]{inputenc}
\usepackage[english]{babel}
\usepackage{csquotes}
\usepackage[shortcuts]{extdash} 

\usepackage[usenames,dvipsnames]{xcolor} 
\usepackage[
 unicode=true,
 pdfusetitle,
 bookmarks=true,
 bookmarksnumbered=false,
 bookmarksopen=false,
 breaklinks=false,
 pdfborder={0 0 1},
 backref=false,
 colorlinks=true,
 linkcolor=blue,
 citecolor=MidnightBlue,
 urlcolor=Violet,
 hypertexnames=false,
 ]
 {hyperref}

\usepackage{amsmath}
\usepackage{amssymb}
\usepackage{amsthm}
\usepackage{thmtools}
\usepackage{mathtools}
\usepackage{esint}
\usepackage{tikz-cd}

\usepackage[
 activate={true,nocompatibility},
 final,
 tracking=true,
 kerning=true,
 spacing=true,
 factor=1100,
 stretch=10,
 shrink=10
]{microtype}
\SetTracking{encoding={*}, shape=sc}{0} 
\microtypecontext{spacing=nonfrench}

\usepackage{lmodern}
\usepackage{mathrsfs} 
\usepackage[mathcal]{euscript} 

\usepackage[
 backend=biber,
 hyperref=true,
 url=false,
 isbn=false,
 eprint=true,
 backref=false,
 style=alphabetic,
 firstinits=true,
 sortcites=true
 ]
 {biblatex}
\renewcommand*{\bibfont}{\footnotesize} 

\numberwithin{equation}{section}
\numberwithin{figure}{section}

\declaretheorem[name=Theorem,style=plain,numberwithin=section]{thm}
\declaretheorem[name=Proposition,style=plain,sibling=thm]{prop}
\declaretheorem[name=Lemma,style=plain,sibling=thm]{lem}
\declaretheorem[name=Corollary,style=plain,sibling=thm]{cor}
\declaretheorem[name=Example,style=definition,sibling=thm,
]{example}
\declaretheorem[name=Definition,style=definition,sibling=thm]{defn}
\declaretheorem[name=Remark,style=remark,sibling=thm]{rem}

\usepackage{cleveref}
\crefname{thm}{Theorem}{Theorems}
\crefname{prop}{Proposition}{Proposition}
\crefname{lem}{Lemma}{Lemmas}
\crefname{cor}{Corollary}{Corollaries}
\crefname{example}{Example}{Examples}
\crefname{defn}{Definition}{Definitions}
\crefname{rem}{Remark}{Remarks}
\crefname{claim}{Claim}{Claims}
\crefname{enumi}{}{}
\crefname{enumii}{}{}
\crefname{enumiii}{}{}
\crefname{equation}{}{}

\usepackage{enumitem}
\newcommand{\enumlabelformat}{\roman}
\newcommand{\enumlabelfont}[1]{#1}
\newlength{\thelabelsep}
\setlist{labelsep=\thelabelsep}
\setlist[enumerate]{font=\enumlabelfont,label=(\enumlabelformat*),leftmargin=2.5em}
\setlist[itemize]{leftmargin=2.5em,label=$-$}

\newcounter{inlineenum}
\renewcommand{\theinlineenum}{\enumlabelformat{inlineenum}}
\newenvironment{inlineenum}
  {\setcounter{inlineenum}{0}%
   \renewcommand{\item}{\refstepcounter{inlineenum}{(\theinlineenum)\hspace{\thelabelsep}}}
  }
  {\ignorespacesafterend}

\input{./macros}

\begin{document}

\title[Essential spectra of tensor product Hilbert complexes]
 {Essential spectra of tensor product Hilbert complexes, and the \texorpdfstring{$\dbar$}{dbar}-Neumann problem on product manifolds}
\author{Franz Berger}
\address{Fakult\"at f\"ur Mathematik, Universit\"at Wien, Oskar-Morgenstern-Platz 1, A-1090 Wien, Austria}
\email{franz.berger2@univie.ac.at}
\thanks{This work was supported by the Austrian Science Fund (FWF): P23664.}

\begin{abstract}
 We investigate tensor products of Hilbert complexes, in particular the (essential) spectrum of their Laplacians.
 It is shown that the essential spectrum of the Laplacian associated to the tensor product complex is computable in terms of the spectra of the factors.
 Applications are given for the $\dbar$-Neumann problem on the product of two or more Hermitian manifolds, especially regarding (non-)~compactness of the associated $\dbar$-Neumann operator.
\end{abstract}


\maketitle
\tableofcontents

\section{Introduction}

By a \emph{Hilbert (cochain) complex} $(H,\mathcal D,d)$ (or simply $(H,d)$) we mean a graded Hilbert space $H = \bigoplus_{i \in \mathbb Z} H_i$ with only finitely many nonzero (mutually orthogonal) terms, a dense graded linear subspace $\mathcal D = \bigoplus_{i\in \mathbb Z} \mathcal D_i$, and a closed linear operator $d \colon \mathcal D \to \mathcal D$ on $H$ of degree $1$ such that $d \circ d = 0$.
We therefore obtain the (cochain) complex
$$ \cdots \hilbcomprightarrow{d_{i-2}} \mathcal D_{i-1} \hilbcomprightarrow{d_{i-1}} \mathcal D_i \hilbcomprightarrow{d_i} \mathcal D_{i+1} \hilbcomprightarrow{d_{i+1}} \cdots $$
with closed and densely defined differentials $d_i \coloneqq d|_{\mathcal D_i} \colon \mathcal D_i \to \mathcal D_{i+1}$.
Hilbert complexes were most prominently studied in \cite{Bruening1992}, but the concept also appears in some earlier works \cite{Vasilescu1980,Grosu1982}.
They are useful in order to formalize the main operator theoretic properties common to boundary value problems for elliptic complexes.
An important operator associated to $(H,d)$ is its \emph{Laplacian}, defined by $\Delta \coloneqq \bigoplus_{i\in \mathbb Z} \Delta_i$ with
$$ \Delta_i \coloneqq d_i^* d_i + d_{i-1}d_{i-1}^* \quad\text{on}\quad \dom{\Delta_i} \coloneqq \big\{x \in \mathcal D_i \cap \mathcal D_i^* : dx \in \mathcal D_{i+1}^* \text{ and } d^*x \in \mathcal D_{i-1}^*\big\}, $$
where $\mathcal D_i^* \subseteq H_i$ is the domain of $d_{i-1}^*$, the adjoint of $d_{i-1}$.
This gives the chain complex
$$ \cdots \hilbcompleftarrow{d_{i-2}^*} \mathcal D_{i-1}^* \hilbcompleftarrow{d_{i-1}^*} \mathcal D_i^* \hilbcompleftarrow{d_i^*} \mathcal D_{i+1}^* \hilbcompleftarrow{d_{i+1}^*} \cdots $$
Each $\Delta_i$ is a positive self-adjoint operator on $H_i$, and it is useful to study the Laplacian in order to gain insight into the solutions of the inhomogeneous $d$-equation.
The fact that the Laplacian is self-adjoint is usually attributed to Gaffney \cite{Gaffney1955}, where the corresponding result for the de~Rham complex is found.

In this article we are concerned with the spectral theory of the Laplacian of a \emph{tensor product} of two Hilbert complexes.
The Hilbert space of the tensor product of two Hilbert complexes $(H,d)$ and $(H',d')$ is given by the tensor product of graded Hilbert spaces, and the differential is the closure of $\bigoplus_{j+k=i}(d_j\otimes \id{H'_k} + \sigma_j\otimes d'_k)$, where $\sigma_j$ is multiplication by $(-1)^j$ on $H_j$, see \cref{sec:hilbert_complex_tensor_products} for the detailed definitions.
Our main result is the following:

\begin{thm}\label{intro_spectrum_and_essspec_of_product_laplacian}
 Let $(H,d)$ and $(H',d')$ be two Hilbert complexes, with Laplacians $\Delta$ and $\Delta'$, respectively.
 If $\widetilde\Delta$ denotes the Laplacian of the tensor product Hilbert complex $(H,d)\hgtensor (H',d')$, then
 \begin{equation}\label{eq:spectrum_of_product_laplacian}
  \spec{\widetilde\Delta_i} = \bigcup_{j+k=i} \big(\spec{\Delta_j} + \spec{\Delta'_k}\big)
 \end{equation}
 and
 \begin{equation}\label{eq:essential_spectrum_of_product_laplacian}
  \essspec{\widetilde\Delta_i} = \bigcup_{j+k=i} \big(\essspec{\Delta_j} + \spec{\Delta'_k}\big)\cup\big(\spec{\Delta_j} + \essspec{\Delta'_k}\big).
 \end{equation}
\end{thm}

Here, $\spec{\widetilde\Delta_i}$ and $\essspec{\widetilde\Delta_i}$ are the spectrum and the essential spectrum of $\widetilde\Delta_i$, respectively, and we use Minkowski sums in order to add sets of real numbers.
In particular, the sum $\essspec{\Delta_j} + \spec{\Delta'_k}$ is meant to be empty if one of the summands is empty.
Equations~\cref{eq:spectrum_of_product_laplacian,eq:essential_spectrum_of_product_laplacian} are obtained by first showing that the Laplacian of the tensor product is an appropriate direct sum of the closures of $\Delta_j \otimes\id{H'_k} + \id{H_j}\otimes \Delta'_k$, and then computing the (essential spectrum) of these operators by using the Borel functional calculus for strongly commuting tuples of normal operators, see \cref{sec:joint_spectra}.

The results are motivated by questions arising in the $\dbar$-Neumann problem on Hermitian manifolds, which is essentially the study of the \emph{(Gaffney extension of the) complex Laplacian},
$$ \square^E \coloneqq \dbar^{E,*}\dbar^E + \dbar^E\dbar^{E,*}, $$
with $E \to X$ a Hermitian holomorphic vector bundle, $\dbar^E$ the (maximal closed extension of) the Dolbeault operator acting on $E$-valued differential forms, and $\dbar^{E,*}$ its Hilbert space adjoint with respect to the $L^2$ inner product induced by the metrics.
Since $\dbar^E$ maps $(p,q)$ forms to $(p,q+1)$ forms and squares to zero, we obtain, for every $1 \leq p \leq \dim[\mathbb C]{X}$, a Hilbert complex which we denote by $(L^2_{p,\bullet}(X,E),\dbar^E)$, with $L^2_{p,q}(X,E)$ being the space of square-integrable $(p,q)$ forms on $X$ with values in $E$.

The Cauchy-Riemann equations on product domains have been studied previously in \cite{Krantz1988,Fu2007,Ehsani2007,Chakrabarti2010,Chakrabarti2011}.
In \cite{Chakrabarti2010}, Chakrabarti computes the spectrum of $\square$ for $X \times Y$, the product of two Hermitian manifolds.
If we denote, for the moment, the complex Laplacian on the $(p,q)$ forms on $X \times Y$ by $\square^{X\times Y}_{p,q}$, then its spectrum according to \cite{Chakrabarti2010} is
\begin{equation*}\label{eq:spectrum_of_product_dbar_chakrabarti}
 \spec{\square^{X\times Y}_{p,q}} = \bigcup_{\substack{p'+p''=p \\ q'+q''=q}} \big(\spec{\square^X_{p',q'}} + \spec{\square^Y_{p'',q''}}\big).
\end{equation*}
One of our goals was to find a similar formula for the essential spectrum.
If we allow for bundle-valued forms, say $E \to X$ and $F \to Y$, then it turns out that the complex Laplacian for the bundle $E\boxtimes F \coloneqq \pi_X^* E \otimes \pi_Y^* F$ over $X \times Y$ is unitarily equivalent to the Laplacian of the tensor product of the Hilbert complexes $(L^2_{0,\bullet}(X,E),\dbar^E)$ and $(L^2_{0,\bullet}(Y,F),\dbar^F)$, so that we obtain
\begin{equation}\label{eq:intro_spectrum_of_product_dbar}
 \spec{\square^{E\boxtimes F}_{0,q}} = \bigcup_{q'+q''=q} \big(\spec{\square^E_{0,q'}} + \spec{\square^F_{0,q''}}\big)
\end{equation}
and
\begin{equation}\label{eq:intro_essential_spectrum_of_product_dbar}
 \essspec{\square^{E\boxtimes F}_{0,q}} = \bigcup_{q'+q''=q} \big(\essspec{\square^E_{0,q'}} + \spec{\square^F_{0,q''}}\big)\cup\big(\spec{\square^E_{0,q'}} + \essspec{\square^F_{0,q''}}\big)
\end{equation}
from \cref{eq:spectrum_of_product_laplacian,eq:essential_spectrum_of_product_laplacian}.
Both equations have their expected analogues for $(p,q)$ forms with $p \neq 0$, but this will require taking an additional direct sum, see \cref{essential_spectrum_of_product_dbar}.

We are also interested in questions regarding the compactness of minimal solution operators to the inhomogeneous $\dbar^E$-equation.
Closely related to this is compactness of the $\dbar$-Neumann operator, which is the inverse of $\square^E$ (modulo its kernel).
Whether the $\dbar$-Neumann operator is compact can be read off from the essential spectrum of $\square^E$, and \cref{eq:intro_essential_spectrum_of_product_dbar} therefore provides a way to decide compactness for product manifolds in terms of the corresponding property of the factors.

We point out that these above questions have already been investigated for certain special product manifolds.
As a standard counterexample, Krantz \cite{Krantz1988} shows that the minimal solution operator to the $\dbar$-equation for $(0,1)$-forms on the unit bidisc in $\mathbb C^2$ fails to be compact.

Haslinger and Helffer consider in \cite[Proposition~4.6]{Haslinger2007} the \emph{weighted} $\dbar$-problem on $\mathbb C^n$, which can be understood as the corresponding problem for the trivial line bundle on $\mathbb C^n$ with nontrivial fiber metric $e^{-\varphi/2}$ for some given smooth function $\varphi \colon \mathbb C^n \to \mathbb R$.
They show that if $\varphi$ is \emph{decoupled}, $\varphi(z) = \varphi_1(z_1) + \dotsb + \varphi_n(z_n)$, and there exist $1 \leq j \leq n$ such that the Bergman space of entire functions on $\mathbb C$, square integrable with respect to $e^{-\varphi_j}\lambda$ (with $\lambda$ the Lebesgue measure), has infinite dimension, then the $\dbar$-Neumann operator for the weighted problem on $\mathbb C^n$ is not compact on $(0,1)$ forms.
The question of whether the conclusion extends to higher degree forms was left unanswered.
Indeed, the method of proof seems unsuitable for treating anything but $(0,1)$ forms, since they basically consider a solution operator for the product complex which only agrees with the minimal one for $(0,1)$ forms, see the arguments in \cite{Chakrabarti2011}.
The deeper reason for this is that the kernel of $\dbar$ does not play nicely with respect to the product structure, while $L^2$ cohomology (the kernel of the Laplacian) does.
This is expressed in the K\"unneth formula (which holds more generally for tensor products of Hilbert complexes, see \cite{Grosu1982} or \cite[Corollary~2.15]{Bruening1992}).
Note that the weighted problem with decoupled weights is covered by our results since, geometrically, it corresponds to considering the line bundle $\bigotimes_{j=1}^n \pi_j^*E_j$ over $\mathbb C^n$, where $E_j$ is the trivial line bundle over $\mathbb C$ with fiber metric $e^{-\varphi_j/2}$, and $\pi_j \colon \mathbb C^n \to \mathbb C$ the projection onto the $j$th factor.

The extension of \cite[Proposition~4.6]{Haslinger2007} will then be \cref{dbar_neumann_noncompactness_riemann_surface_products}, where we show that the $\dbar$-Neumann operator for the product of $n$ Riemann surfaces (and vector bundles over them) is in fact not compact on $(0,q)$ forms with $0 \leq q \leq n-1$, provided at least one factor has an infinite dimensional Bergman space.

\subsubsection*{Acknowledgments}
This work is part of the author's Ph.D.\ research under the supervision of Prof.\ Friedrich Haslinger.
The author would like to thank Prof.\ Haslinger for making him aware of the above mentioned difficulties in the weighted $\dbar$-problem with decoupled weights, which ultimately led to the present article.

\section{Hilbert complexes}\label{sec:hilbert_complexes}

In this section we will review some of the basics of the theory of Hilbert complexes.
For a more in-depth introduction, see \cite{Bruening1992}.
In addition, we will supplement this by adding concepts and results which are standard in the $L^2$ theory of the $\dbar$-complex from several complex variables.

If $(H,\mathcal D,d)$ and $(H',\mathcal D',d')$ are Hilbert complexes, then a graded linear map $g \colon H \to H'$ (of degree $0$) is called a \emph{morphism of Hilbert complexes} if $g$ is bounded (i.e., $g_i \coloneqq g|_{H_i} \colon H_i \to H'_i$ is bounded for every $i \in \mathbb Z$) and $g \circ d \subseteq d' \circ g$.
In particular, $g(\mathcal D) \subseteq \mathcal D'$.
An isomorphism of Hilbert complexes is a bijective morphism of Hilbert complexes $g \colon H \to H'$ such that $g\circ d = d' \circ g$ (in the sense of unbounded operators; in particular, $g(\mathcal D) = \mathcal D'$).
A \emph{unitary equivalence} between Hilbert complexes is a unitary isomorphism of Hilbert complexes.
Note that if $g \colon (H,d) \to (H',d')$ is such a unitary equivalence, then $d^* \circ g^{-1} = g^{-1} \circ d'^*$ and hence $g \circ \Delta = \Delta' \circ g$, so that the Laplacians are unitarily equivalent.

If $\{(H^j,d^j) : j \in F\}$ is a finite collection of Hilbert complexes, then their \emph{direct sum} is the Hilbert complex $\bigoplus_{j \in F} (H^j,d^j) \coloneqq (\bigoplus_{j \in F} H^j, \bigoplus_{j \in F} d^j)$.
Evidently, its Laplacian is given by $\bigoplus_{j \in F} \Delta^j$, with $\Delta^j$ the Laplacian of $(H^j,d^j)$.

The \emph{cohomology} of a Hilbert cochain complex $(H,d)$ is the graded vector space
$$ \cohom{H,d} \coloneqq \bigoplus_{i\in \mathbb Z} \cohom[i]{H,d}, \quad\text{where}\quad \cohom[i]{H,d} \coloneqq \ker(d_i)\mathbin{\big/}\img{d_{i-1}}. $$
The \emph{reduced cohomology} of $(H,d)$ is
$$ \redcohom{H,d} \coloneqq \bigoplus_{i\in \mathbb Z} \redcohom[i]{H,d}, \quad\text{where}\quad \redcohom[i]{H,d} \coloneqq \ker(d_i)\mathbin{\big/}\overline{\img{d_{i-1}}}. $$
In general, the differentials of a Hilbert complex do not have closed range, so that typically only $\redcohom{H,d}$ will be a Hilbert space in a natural way.
One of the main tools available is the Hodge decomposition, see \cite[Lemma~2.1]{Bruening1992}:

\begin{prop}[Weak Hodge decomposition]
 Every Hilbert complex $(H,d)$ induces an orthogonal decomposition
 $$ H_i = \ker(\Delta_i) \oplus \overline{\img{d_{i-1}}} \oplus \overline{\img{d_i^*}}. \quad (i \in \mathbb Z) $$
 Moreover, the space of \emph{harmonic elements},
 $$ \ker(\Delta) = \bigoplus_{i \in \mathbb Z} \big(\ker(d_i) \cap \ker(d_{i-1}^*)\big), $$
 is canonically isomorphic to $\redcohom{H,d}$, in the sense that every equivalence class in $\redcohom{H,d}$ has a unique harmonic representative.
\end{prop}

Let $P^d \colon H \to H$ denote the orthogonal projection of $H$ onto $\ker(d)$.
The \emph{minimal (or canonical) solution operator to $(H,d)$} is the closed operator
$$ S = S(H,d) \colon \img{d} \subseteq H \to H, \quad S(dx) \coloneqq (\id{H}-P^d)x. $$
This is well-defined since $\ker(d) = \ker(\id{H} - P^d)$.
We write $S_i = S_i(H,d) \colon \img{d_{i-1}} \to H_{i-1}$ for its restriction to $H_i$.
By definition, $S$ gives the norm-minimal solution to the inhomogeneous $d$-equation,
$$ d(Sy) = y \quad\text{and}\quad Sy \perp \ker(d) $$
for $y \in \img{d}$.

The remaining results of this section are well-known for the (closed extensions of the) Dolbeault complex on Hermitian manifolds.
As a (non-exhaustive) list of references, we cite \cite{Hoermander1965,Chen2001,Straube2010,Haslinger2014}.
For the convenience of the reader, we provide here the proofs of the corresponding results for Hilbert complexes.
Note that while most of those references do not allow for $\Delta_i$ to have a nontrivial kernel (since the complex Laplacian on bounded pseudoconvex domains is injective), this is easily incorporated into the arguments, see also \cite{Ovrelid2014,Ruppenthal2011}.

\begin{lem}\label{minmal_solution_operator_bounded_iff_diff_range_closed}
 Let $(H,d)$ be a Hilbert complex.
 Then $S_i \colon \img{d_{i-1}} \subseteq H_i \to H_{i-1}$ is bounded if and only if $d_{i-1}$ has closed range.
 In this case we extend $S_i$ to $H_i$ by zero on $\img{d_{i-1}}^\perp$.
\end{lem}

\begin{proof}
 If $d_{i-1}$ has closed range, then $S_i$ is a closed and everywhere defined operator on the Hilbert space $\img{d_{i-1}}$, hence bounded by the closed graph theorem.
 Conversely, if $S_i$ is bounded there exists $C > 0$ such that $\|S(dx)\| \leq C \|dx\|$ for all $x \in \mathcal D_{i-1}$.
 If $x \in \mathcal D_{i-1} \cap \ker(d_{i-1})^\perp$, then $S(dx) = x$ and hence
 $\|x\| = \|S(dx)\| \leq C \|dx\|$,
 which shows that $d_{i-1}$ has closed range.
\end{proof}

The next result shows that the minimal solution operator is closely related to the Laplacian:

\begin{prop}\label{inverse_of_laplacian_properties}
 Let $(H,d)$ be a Hilbert complex and define
 $$ N = N(H,d) \coloneqq \big(\Delta|_{\dom{\Delta} \cap \ker(\Delta)^\perp}\big)^{-1} \colon \img{\Delta} \to H $$
 as the inverse of the Laplacian.
 We write $N_i = N_i(H,d) \colon \img{\Delta_i} \to H_i$ for its restriction to $H_i$.
 Then:
 \begin{enumerate}
  \item\label{item:N_commutes_with_differentials} $dN=Nd$ on $\mathcal D \cap \img{\Delta}$ and $d^*N = Nd^*$ on $\mathcal D^* \cap \img{\Delta}$.
  \item On $\img{d} \cap \img{\Delta}$ we have
   \begin{equation}\label{eq:minimal_solution_operator_in_terms_of_N}
    S = d^* N.
   \end{equation}
  \item On $\mathcal D \cap d^{-1}(\img{\Delta})$ we have
   \begin{equation}\label{eq:projection_onto_kernel_in_terms_of_N}
    I - P^d = d^* N d.
   \end{equation}
 \end{enumerate}
\end{prop}

\begin{proof}
 If $x \in \mathcal D_i \cap \img{\Delta_i}$, then $x = \Delta_i y$ for some $y \in \dom{\Delta_i} \cap \ker(\Delta_i)^\perp$.
 It follows that $d_iy \in \dom{\Delta_{i+1}}$ and
 $$ Nd_ix = Nd_i\Delta_i y = N d_i d_i^* d_i y = N(d_{i+1}^* d_{i+1} + d_id_i^*)d_i y = d_i y = d_i Nx. $$
 This shows the first equation in~\cref{item:N_commutes_with_differentials}, the other one follows similarly.
 If $x \in \img{d_{i-1}} \cap \img{\Delta_i}$, then
 $$ x = \Delta_i N_i x = d_i^*d_i N_i x + d_{i-1}d_{i-1}^* N_i x. $$
 Because $x \in \ker(d_i)$ and $d_i \circ d_{i-1} = 0$, this implies $d_i^*d_iN_ix \in \ker(d_i) \cap \img{d_i^*} = 0$.
 Therefore, $x = d_{i-1}d_{i-1}^* N_i x$ and
 $$ S_i x = (I-P^d)d_{i-1}^*N_ix = d_{i-1}^* N_i x $$
 since $\overline{\img{d_{i-1}^*}} = \img{I-P^d} \cap H_{i-1}$.
 This shows \cref{eq:minimal_solution_operator_in_terms_of_N}, and \cref{eq:projection_onto_kernel_in_terms_of_N} is immediate from the definition of $S$.
\end{proof}

\begin{lem}\label{equivalent_statements_to_N_bounded}
 Let $(H,d)$ be a Hilbert complex.
 Then the following are equivalent:
 \begin{enumerate}
  \item\label{item:equivalent_statements_to_N_bounded_N_i_bounded}
   $N_i \colon \img{\Delta_i} \to H_i$ is bounded.
  \item\label{item:equivalent_statements_to_N_bounded_Delta_closed_range}
   $\Delta_i$ has closed range.
  \item\label{item:equivalent_statements_to_N_bounded_diff_closed_range}
   $d_{i-1}$ and $d_i$ both have closed range.
  \item\label{item:equivalent_statements_to_N_bounded_basic_estimate}
   There is $C>0$ such that, for all $x \in \mathcal D_i \cap \mathcal D_i^* \cap \ker(\Delta_i)^\perp$,
   $$ \|x\|^2 \leq C\big(\|d_i x\|^2 + \|d_{i-1}^* x\|^2\big). $$
  \item\label{item:equivalent_statements_to_N_bounded_S_closed_range}
   $S_i \colon \img{d_{i-1}} \to H_{i-1}$ and $S_{i+1} \colon \img{d_i} \to H_i$ are both bounded.
 \end{enumerate}
 In this case, we extend $N_i$ by zero on $\img{\Delta_i}^\perp = \ker(\Delta_i)$.
\end{lem}

\begin{proof}
 Because $N_i$ is closed, \cref{item:equivalent_statements_to_N_bounded_Delta_closed_range} $\implies$ \cref{item:equivalent_statements_to_N_bounded_N_i_bounded}.
 Conversely, suppose $N_i$ is bounded and take $u_j \to u$ with $u_j \in \img{\Delta_i}$.
 Then $N_i u_j \to v$ for some $v \in H$ and we have $\Delta_i N_i u_j = u_j$.
 As $\Delta_i$ is closed, $v \in \dom{\Delta_i}$ and $\Delta_i v = u$, hence $u \in \img{\Delta_i}$.
 Thus, \cref{item:equivalent_statements_to_N_bounded_N_i_bounded} $\iff$ \cref{item:equivalent_statements_to_N_bounded_Delta_closed_range}.
 
 We now show \cref{item:equivalent_statements_to_N_bounded_Delta_closed_range} $\implies$ \cref{item:equivalent_statements_to_N_bounded_diff_closed_range}, so assume that $\Delta_i$ has closed range.
 For $x \in \mathcal D_i \cap \ker(d_i)^\perp \subseteq \ker(\Delta_i)^\perp = \img{\Delta_i}$, we have
 $$ \|x\|^2 = \langle \Delta_i N_i x,x \rangle = \langle d_i^*d_i N_i x,x \rangle + \langle d_{i-1}d_{i-1}^* N_i x, x \rangle = \langle d_i N_i x, d_i x \rangle \leq C \|x\| \|d_i x\| $$
 because $d_{i-1}d_{i-1}^* N_i x \in \ker(d_i) \perp x$, and the operators $d_i N_i$ and $d_{i-1}^* N_i$ are bounded on $\img{\Delta_i}$ since
 $$ \|d_i N_i y\|^2 + \|d_{i-1}^* N_i y\|^2 = \langle \Delta_i N_i y,N_i y \rangle = \langle y, N_i y\rangle \quad (y \in \img{\Delta_i}) $$
 and $N_i$ is bounded by \cref{item:equivalent_statements_to_N_bounded_N_i_bounded}.
 Therefore, $d_i$ has closed range.
 Interchanging the roles of $d_i$ and $d_{i-1}^*$, one shows that the latter operator also has closed range.
 
 Now assume that $d_{i-1}$ and $d_i$ have closed range. 
 It follows that $d_i^*$ also has closed range. 
 If $x \in \mathcal D_i \cap \mathcal D_i^* \cap \ker(\Delta_i)^\perp = \mathcal D_i \cap \mathcal D_i^* \cap (\img{d_{i-1}} \oplus \img{d_i^*})$, write $x = x_1 + x_2$ with $x_1 \in \mathcal D_i \cap \mathcal D_i^* \cap \img{d_{i-1}}$ and $x_2 \in \mathcal D_i \cap \mathcal D_i^* \cap \img{d_i^*}$.
 There exist $C_1,C_2>0$ such that
 $$ \|x_1\|^2 \leq C_1 \|d_{i-1}^*x_1\|^2 = C_1 \|d_{i-1}^* x\|^2 \quad\text{and}\quad \|x_2\|^2 \leq C_2 \|d_i x_2\|^2 = C_2 \|d_i x\|^2 $$
 by our assumptions on $d_{i-1}$ and $d_i$, and hence
 $$ \|x\|^2 = \|x_1\|^2 + \|x_2\|^2 \leq C\big(\|d_{i-1}^*x\|^2 +\|d_i x\|^2\big) $$
 with $C \coloneqq \max\{C_1,C_2\}$.
 This shows \cref{item:equivalent_statements_to_N_bounded_diff_closed_range} $\implies$ \cref{item:equivalent_statements_to_N_bounded_basic_estimate}, and \cref{item:equivalent_statements_to_N_bounded_basic_estimate} $\implies$ \cref{item:equivalent_statements_to_N_bounded_Delta_closed_range} is immediate as $\dom{\Delta_i} \subseteq \mathcal D_i \cap \mathcal D_i^*$ and $\|d_ix\|^2 + \|d_{i-1}^* x\|^2 = \langle \Delta_i x,x \rangle \leq \|\Delta_i x\| \|x\|$ for $x \in \dom{\Delta_i}$.
 The equivalence \cref{item:equivalent_statements_to_N_bounded_diff_closed_range} $\iff$ \cref{item:equivalent_statements_to_N_bounded_S_closed_range} follows from \cref{minmal_solution_operator_bounded_iff_diff_range_closed}.
\end{proof}

\begin{prop}\label{inverse_of_laplacian_properties_bounded}
 Let $(H,d)$ be a Hilbert complex, and assume that any of the equivalent statements of \cref{equivalent_statements_to_N_bounded} holds.
 We extend $N_i$ by zero on $\img{\Delta_i}^\perp$.
 Then
 \begin{enumerate}
  \item\label{item:N_commutes_with_differentials_2}
   $dN = Nd$ on $\mathcal D_i$ and $d^*N = Nd^*$ on $\mathcal D_i^*$,
  \item\label{item:S_in_terms_of_N}
   $S_i = d^*N_i$ on $H_i$,
  \item\label{item:N_in_terms_of_S}
   $N_i = S_i^* S_i + S_{i+1}S_{i+1}^*$, where $S_i$ and $S_{i+1}$ are the extensions by zero,
  \item\label{item:N_bounded_zero_in_spectrum}
   $\inf{\spec{\Delta_i}} > 0$ if and only if $\ker(\Delta_i) = 0$, and
  \item\label{item:N_bounded_zero_in_essential_spectrum}
   $\inf{\essspec{\Delta_i}} > 0$ if and only if $\dim{\ker(\Delta_i)} < \infty$.
 \end{enumerate}
\end{prop}

\begin{proof}
 On $\mathcal D_i \cap \img{\Delta_i}^\perp \subseteq \ker(d_i)$ we have $dN = 0$ and $Nd = 0$, so $dN = Nd$ holds on $\mathcal D_i$, and similarly one shows $d^*N=Nd^*$ on $\mathcal D_i^*$.
 
 By the Hodge decomposition, $\img{\Delta_i} = \img{d_{i-1}} \oplus \img{d_i^*}$.
 Thus, $\img{d_{i-1}} \subseteq \img{\Delta_i}$ and hence $S = d^*N$ on $\img{d_{i-1}}$.
 Since $S|_{\img{d_{i-1}}^\perp} = 0$ by definition, it remains to show that $d^*N$ also vanishes on $\img{d_{i-1}}^\perp$.
 As $\img{d_{i-1}}^\perp = \ker(\Delta_i) \oplus \img{d_i^*}$ and $N|_{\ker(\Delta_i)} = 0$, we are left with showing that $d^*N|_{\img{d_i^*}} = 0$.
 Now if $y \in \mathcal D_{i+1}^* = \dom{d_i^*}$, then $d^* N d_i^* y = d^* d_i^* N y = 0$ by \cref{item:N_commutes_with_differentials_2}.
 This shows that $S = d^*N$ on $H_i$.
 
 We have $d_{i-1}^*N_i x \in \mathcal D_{i-1}$ and $d_iN_i x \in \mathcal D_{i+1}^*$ for $x \in H_i$, and therefore
 $$ N_i x = N_i \Delta_i N_i x = (N_i d_{i-1})(d_{i-1}^*N_i)x + (N_i d_i^*)(d_i N_i)x = (d_{i-1}N_i)(d_{i-1}^*N_i)x + (d_i^* N_i)(d_iN_i)x $$
 by \cref{item:N_commutes_with_differentials_2}.
 Applying $S_i^* = d_{i-1} N_{i-1}$ and $S_{i+1}^* = d_i N_i$ shows \cref{item:N_in_terms_of_S}.

 The orthogonal decomposition $H_i = \ker(\Delta_i) \oplus \img{\Delta_i}$ induces a unitary equivalence of $\Delta_i$ with
 $$ 0 \oplus \Delta_i|_{\img{\Delta_i}} \colon \ker(\Delta_i) \oplus \img{\Delta_i} \to \ker(\Delta_i) \oplus \img{\Delta_i}, $$
 hence $\spec{\Delta_i}\setminus\{0\} = \spec{\Delta_i|_{\img{\Delta_i}}}$ and $\essspec{\Delta_i}\setminus\{0\} = \essspec{\Delta_i|_{\img{\Delta_i}}}$ since $0 \not\in \spec{\Delta_i|_{\img{\Delta_i}}}$ by the boundedness of $N_i|_{\img{\Delta_i}}$.
 Moreover, $0 \in \spec{\Delta_i}$ (resp.\ $0 \in \essspec{\Delta_i}$) if and only if $\ker(\Delta_i) \neq 0$ (resp.\ $\dim{\ker(\Delta_i)} = \infty$).
 This immediately gives \cref{item:N_bounded_zero_in_spectrum,item:N_bounded_zero_in_essential_spectrum}.
\end{proof}

\begin{rem}
 Concerning items~\cref{item:N_bounded_zero_in_spectrum,item:N_bounded_zero_in_essential_spectrum}, one even has that $\inf{\spec{\Delta_i}} > 0$ (resp.\ $\inf{\essspec{\Delta_i}} > 0$) if and only if the conditions of \cref{equivalent_statements_to_N_bounded} are satisfied and $\ker(\Delta_i) = 0$ (resp.\ $\dim{\ker(\Delta_i)} < \infty$).
 This is Proposition~2.2 (resp.\ Proposition~2.3) of \cite{Fu2010}.
\end{rem}

We are interested in determining whether $N$ and $S$ are compact operators.
Recall that the \emph{essential spectrum} $\essspec{T}$ of a (say normal) operator $T$ is the set of complex numbers which are accumulation points of the spectrum or eigenvalues of infinite multiplicity.
We refer to \cref{sec:joint_spectra} for the precise definition and more information on $\essspec{T}$.

\begin{prop}\label{characterization_of_compactness_of_S}
 Let $(H,d)$ be a Hilbert complex and assume that $i \in \mathbb Z$ is such that any of the equivalent statements of \cref{equivalent_statements_to_N_bounded} holds.
 Then the following are equivalent:
 \begin{enumerate}
  \item $N_i \colon H_i \to H_i$ is compact.
  \item $S_i \colon H_i \to H_{i-1}$ and $S_{i+1} \colon H_{i+1} \to H_i$ are both compact.
  \item $\essspec{\Delta_i} \subseteq \{0\}$.
 \end{enumerate}
\end{prop}

\begin{proof}
 If $S_i$ and $S_{i+1}$ are compact (in particular: bounded), then \cref{item:N_in_terms_of_S} of \cref{inverse_of_laplacian_properties_bounded} shows that $N_i$ is also compact.
 Conversely, $S_i$ and $S_{i+1}$ are compact as soon as $N_i$ is since both $S_i^* S_i$ and $S_{i+1}S_{i+1}^*$ are positive operators.
 Indeed, if $A$ and $B$ are bounded positive operators on a Hilbert space $K$ with $A \leq B$ and $B$ compact, then the compact operator $B^{1/2}$ satisfies
 $$ \|A^{1/2}x\|^2 = \langle A x,x \rangle \leq \langle B x,x \rangle = \|B^{1/2}x\|^2 $$
 for every $x \in K$.
 Since $B^{1/2}x_j \to 0$ in $K$ for every weak null sequence $x_j$, we see that $A^{1/2}$ is compact, and hence $A$ is also compact.
 Now apply this to $S_i^*S_i \leq N_i$ and $S_{i+1}S_{i+1}^* \leq N_i$.

 We know that $\essspec{\Delta_i} \setminus \{0\} = \essspec{\Delta_i|_{\img{\Delta_i}}}$, see the proof of item~\cref{item:N_bounded_zero_in_essential_spectrum} of \cref{inverse_of_laplacian_properties_bounded}.
 But $N_i$ is compact if and only if $N_i|_{\img{\Delta_i}}$ is, and this is the case if and only if $\Delta_i|_{\img{\Delta_i}}$ has compact resolvent, i.e., $\essspec{\Delta_i|_{\img{\Delta_i}}} = \emptyset$. 
\end{proof}

\section{Tensor products of Hilbert complexes}\label{sec:hilbert_complex_tensor_products}

Having dealt with the basics on Hilbert complexes, we can now move on to their tensor products.
If $H$ and $K$ are Hilbert spaces, then we denote by $H\htensor K$ their Hilbert space tensor product, which is the completion of the algebraic tensor product $H\otimes K$ with respect to the usual inner product.
We require a few basic facts about the tensor product of unbounded operators, see \cite[7.5]{Schmuedgen2012} for a reference.
If $T$ and $S$ are closable linear operators on $H$ and $K$, respectively, then the induced operators $T \otimes S$ and $T \otimes \id{H} + \id{K} \otimes S$ on $\dom{T} \otimes \dom{S} \subseteq H \htensor K$ are closable.
We denote the closure of $T\otimes S$ by $T\htensor S$.
If both $T$ and $S$ are densely defined and closable, then $(T \htensor S)^* = T^* \htensor S^*$.

For two $\mathbb Z$-graded vector spaces $A = \bigoplus_{i \in \mathbb Z} A_i$ and $B = \bigoplus_{i \in \mathbb Z} B_i$, we denote by $A\gtensor B$ their \emph{graded tensor product}, which is the graded vector space
\begin{equation}\label{eq:graded_vector_space_tensor_product}
 A \gtensor B = \bigoplus_{i \in \mathbb Z} (A \gtensor B)_i \quad\text{with}\quad (A\gtensor B)_i \coloneqq \bigoplus_{j+k=i} A_j \otimes B_k.
\end{equation}
If $H$ and $K$ are $\mathbb Z$-graded Hilbert spaces, and if only finitely many $H_i$ and $K_i$ are nonzero, then we write $H\hgtensor K$ for the \emph{tensor product of graded Hilbert spaces},
$$ H \hgtensor K = \bigoplus_{i \in \mathbb Z} (H \hgtensor K)_i \quad\text{with}\quad (H\hgtensor K)_i \coloneqq \bigoplus_{j+k=i} H_j \htensor K_k. $$
If $A_i$ with $i \in \mathbb Z$ is a sequence of vector spaces, then by $A_\bullet$ we mean the graded vector space $\bigoplus_{i \in \mathbb Z} A_i$.
In the case where $A_i$ is only defined for a subset of $\mathbb Z$, we extend this sequence by zero.
We use the same convention for (finitely many) Hilbert spaces, graded vector bundles and sequences of linear operators.
Finally, the tensor product of Hilbert complexes is defined as in \cite{Bruening1992}:

\begin{defn}\label{def:product_hilbert_complex}
 Given two Hilbert complexes $(H,\mathcal D,d)$ and $(H',\mathcal D',d')$, their \emph{tensor product} complex $(H \hgtensor H', d\hgtensor d')$ is given by the tensor product of graded Hilbert spaces and
 $(d\hgtensor d')_i$ is the closure of
 \begin{equation}\label{eq:product_hilbert_complex_differential}
  \bigoplus_{j+k=i} \big(d_j \otimes \id{H'_k} + \sigma_j \otimes d'_k\big) \colon (\mathcal D\gtensor \mathcal D')_i \to (\mathcal D\gtensor \mathcal D')_{i+1},
 \end{equation}
 where
 $\sigma_j \colon H_j \to H_j$ is the multiplication by $(-1)^j$.
 It is straightforward to verify that this again defines a Hilbert complex.
 Note that the domain of $d\hgtensor d'$ is, in general, strictly larger than $\mathcal D \gtensor \mathcal D'$.
 We denote this tensor product complex by $(H,d) \hgtensor (H',d') \coloneqq (H\hgtensor H', d\hgtensor d')$.
\end{defn}

\begin{prop}\label{laplacian_of_product_hilbert_complex}
 Let $(H,d)$ and $(H',d')$ be two Hilbert complexes, $\Delta$ and $\Delta'$ their respective Laplacians.
 \begin{enumerate}
  \item\label{item:laplacian_of_product_hilbert_complex}
   The Laplacian of $(H,d)\hgtensor(H',d')$ on $(H\hgtensor H')_i$ is the closure of
   \begin{equation}\label{eq:laplacian_of_product_hilbert_complex}
    \bigoplus_{j+k=i} \big(\Delta_j \otimes \id{H'_k} + \id{H_j}\otimes \Delta'_k\big) \colon (\dom{\Delta} \gtensor \dom{\Delta'})_i \to (H \hgtensor H')_i.
   \end{equation}
  \item\label{item:laplacian_of_product_hilbert_complex_closed_range}
   If both $d$ and $d'$ have closed range, then so does $d\hgtensor d'$.
 \end{enumerate}
\end{prop}

\begin{proof}
 By general principles, $(d\hgtensor d')^*$ is the adjoint of the operator \cref{eq:product_hilbert_complex_differential}.
 It follows that
 $$ (d\hgtensor d')^*_i \supseteq
     \bigoplus_{j+k=i} \big(d_j^* \otimes \id{H'_k} + \sigma_j \otimes d'^*_k\big). $$
 If $\widetilde\Delta$ denotes the Laplacian of the tensor product complex, then this gives
 \begin{multline*}
  \widetilde\Delta_i = (d \hgtensor d')_i^* (d\hgtensor d')_i + (d\hgtensor d')_{i-1}(d\hgtensor d')_{i-1}^* \supseteq \\
   \supseteq \bigg(\bigoplus_{j+k=i} \big(d_j^* \otimes \id{H'_k} + \sigma_j \otimes d'^*_k\big)\bigg)\bigg(\bigoplus_{j+k=i} \big(d_j \otimes \id{H'_k} + \sigma_j \otimes d'_k\big)\bigg) + \\
   + \bigg(\bigoplus_{j+k=i-1} \big(d_j \otimes \id{H'_k} + \sigma_j \otimes d'_k\big)\bigg)\bigg(\bigoplus_{j+k=i-1} \big(d_j^* \otimes \id{H'_k} + \sigma_j \otimes d'^*_k\big)\bigg)
 \end{multline*}
 and the latter operator is an extension of
 $$ \bigoplus_{j+k=i} \big(\Delta_j \otimes \id{H'_k} + \id{H_j}\otimes \Delta'_k + (d^*_{j-1}\sigma_j)\otimes d'_k + (\sigma_{j+1} d_j) \otimes d'^*_{k-1} + (d_j\sigma_j)\otimes d'^*_{k-1} + (\sigma_{j-1} d^*_{j-1}) \otimes d'_k\big). $$
 Since $\sigma_{j+1} d_j = - d_j\sigma_j$ and $\sigma_{j-1}d^*_{j-1} = - d^*_{j-1}\sigma_j$, the cross terms vanish, and because the domain of the whole $i$th component is $\dom{\Delta_j}\otimes \dom{\Delta'_k}$, the whole expression is equal to the operator \cref{eq:laplacian_of_product_hilbert_complex} with domain $\bigoplus_{j+k=i} \dom{\Delta_j} \otimes \dom{\Delta'_k}$.
 It is a general fact that for self-adjoint operators $T$ and $S$ on Hilbert spaces $H$ and $K$, respectively, the operator $T \otimes \id{K} + \id{H} \otimes S$ is essentially self-adjoint, see \cite[Theorem~VIII.33]{Reed1980}.
 By the above, $\widetilde \Delta_i$ is a self-adjoint extension of \cref{eq:laplacian_of_product_hilbert_complex} and must therefore equal its closure.
 This shows \cref{item:laplacian_of_product_hilbert_complex}.
 For the proof of \cref{item:laplacian_of_product_hilbert_complex_closed_range} we refer to \cite[Corollary~2.15]{Bruening1992} or \cite[Theorem~4.5]{Chakrabarti2011}.
\end{proof}

Using \cref{laplacian_of_product_hilbert_complex} and the results on the spectra of the (closures of the) operators $\Delta_j \otimes \id{H'_k} + \id{H_j} \otimes \Delta'_k$ from \cref{sec:joint_spectra}, we are now able to show \cref{intro_spectrum_and_essspec_of_product_laplacian}:


\begin{proof}[Proof of \cref{intro_spectrum_and_essspec_of_product_laplacian}]
 The spectrum of the direct sum of finitely many self-adjoint operators decomposes as the union of the spectra of the individual operators, and the same holds for the essential spectrum.
 Now \cref{eq:spectrum_of_product_laplacian,eq:essential_spectrum_of_product_laplacian} follow from \cref{laplacian_of_product_hilbert_complex,spectrum_of_T_plus_S}.
\end{proof}

\begin{rem}
 Due to our choice of having Hilbert complexes $\mathbb Z$-graded and with $H_i = 0$ for $|i|$ large, it may appear at first glance that there are contributions of many ``zero'' operators in \cref{eq:spectrum_of_product_laplacian} and \cref{eq:essential_spectrum_of_product_laplacian}, simply by choosing $j$ and $k$ large enough and with opposite sign.
 This is not an issue since those zero operators act on the zero Hilbert space, so they are invertible and therefore have \emph{empty} spectrum (and not $\{0\}$!).
 In fact, in \cref{eq:spectrum_of_product_laplacian,eq:essential_spectrum_of_product_laplacian}, only the terms with $j \in \supp(H,d)$ and $k \in \supp(H',d')$ contribute, where the \emph{support} of a Hilbert complex $(H,d)$ is the finite set
 $$ \supp(H,d) \coloneqq \{i \in \mathbb Z : H_i \neq 0\} = \{i \in \mathbb Z : \mathcal D_i \neq 0. \}. $$
 Evidently, $\supp((H,d)\hgtensor (H',d')) = \supp(H,d) + \supp(H',d')$.
\end{rem}

We next give a characterization for the compactness of $N$ for the tensor product complex by using formula \cref{eq:essential_spectrum_of_product_laplacian}.
This characterization is simpler and more insightful if the Hilbert complexes are nondegenerate in the following sense:

\begin{defn}\label{def:nondegenerate_hilbert_complex}
 A Hilbert complex $(H,d)$ will be called \emph{nondegenerate} if $d_{i-1} \neq 0$ or $d_i \neq 0$ for all $i \in \supp(H,d)$.
\end{defn}

\begin{lem}\label{hilbert_complex_nondegeneracy}
 Let $(H,d)$ be a nondegenerate Hilbert complex.
 Then $\spec{\Delta_i} \not\subseteq\{0\}$ for all $i \in \supp(H,d)$, i.e., $\spec{\Delta_i}$ is not empty and also not the singleton $\{0\}$.
\end{lem}

\begin{proof}
 Let $i \in \supp(H,d)$.
 We have $\Delta_i = 0$ if and only if $d_i = 0$ and $d_{i-1}^* = 0$.
 Indeed, if $\Delta_i = 0$, then $\dom{\Delta_i} = \ker(\Delta_i) = H_i$ and $\ker(\Delta_i) = \ker(d_i) \cap \ker(d_{i-1}^*)$, so $d_i = 0$ and $d_{i-1}^* = 0$.
 The other implication is obvious.
 Since the differentials are densely defined and closed, this is equivalent to $d_i = 0$ and $d_{i-1} = 0$.
 But if $d_i=d_{i-1} = 0$, then $\mathcal D_i = 0$ by our non-degeneracy assumption, a contradiction to $i \in \supp(H,d)$.
 Therefore, $\Delta_i \neq 0$.
 Since $H_i \neq 0$, we have $\spec{\Delta_i} \neq \emptyset$.
 If $\spec{\Delta_i} = \{0\}$, then $\supp(P_i) = \{0\}$ with $P_i$ the spectral measure associated to $\Delta_i$ as in the spectral theorem, and hence $\Delta_i = \int_{\{0\}} \id{\mathbb R}\,dP_i = 0$, a contradiction.
 It follows that $\spec{\Delta_i} \neq \{0\}$.
\end{proof}

\begin{thm}\label{characterization_of_compactness_of_S_for_product_complex}
 Let $(H,d)$ and $(H',d')$ be two Hilbert complexes, with Laplacians $\Delta$ and $\Delta'$, respectively.
 Assume that $d$ and $d'$ have closed range (in all degrees).
 Denote by $N$ the inverse of the Laplacian for $(H,d)\hgtensor (H',d')$ as in \cref{inverse_of_laplacian_properties}.
 Then the following are equivalent:
 \begin{enumerate}
  \item\label{item:characterization_of_compactness_of_S_for_product_complex_N_compact}
   $N_i \colon (H\hgtensor H')_i \to (H\hgtensor H')_i$ is a compact operator.
  \item\label{item:characterization_of_compactness_of_S_for_product_complex_N_compact_restrictions}
   $N_i|_{H_j \htensor H'_k} \colon H_j \htensor H'_k \to H_j \htensor H'_k$ is a compact operator for all $j,k \in \mathbb Z$ with $j+k=i$.
  \item\label{item:characterization_of_compactness_of_S_for_product_complex_trivial}
   $\essspec{\Delta_j} + \spec{\Delta'_k} \subseteq \{0\}$ and $\spec{\Delta_j} + \essspec{\Delta'_k} \subseteq \{0\}$ for all $j,k \in \mathbb Z$ with $j+k=i$.
 \end{enumerate}
 If, in addition, $(H,d)$ and $(H',d')$ are nondegenerate, then the above are also equivalent to:
 \begin{enumerate}
  \setcounter{enumi}{3}
  \item\label{item:characterization_of_compactness_of_S_for_product_complex_nondegenerate}
   $\essspec{\Delta_j} = \essspec{\Delta'_k} = \emptyset$ for all $j \in \supp(H,d)$ and $k \in \supp(H',d')$ with $j+k=i$.
  \item\label{item:characterization_of_compactness_of_S_for_product_complex_nondegenerate_essential_spectrum_empty}
   $\essspec{\widetilde\Delta_i} = \emptyset$, where $\widetilde\Delta$ is the Laplacian for the tensor product complex.
  \item\label{item:characterization_of_compactness_of_S_for_product_complex_nondegenerate_compactness_of_factors}
   For all $j \in \supp(H,d)$ and $k \in \supp(H',d')$ with $j+k=i$,
   $$ \dim{\cohom[j]{H,d}} < \infty \quad\text{and}\quad \dim{\cohom[k]{H',d'}} < \infty, $$
   and the operators
   $$ N_j(H,d) \colon H_j \to H_j \quad\text{and}\quad N_k(H',d') \colon H'_k \to H'_k $$
   are compact.
 \end{enumerate}
\end{thm}

\begin{proof}
 From \cref{laplacian_of_product_hilbert_complex} we know that $d\hgtensor d'$ has closed range, hence $N_i$ is a bounded operator for all $i \in \mathbb Z$ by \cref{equivalent_statements_to_N_bounded}.
 By \cref{characterization_of_compactness_of_S} and \cref{eq:essential_spectrum_of_product_laplacian}, $N_i$ is compact if and only if
 \begin{equation}\label{eq:proof_characterization_of_compactness_of_S_for_product_complex_1}
  \essspec{\Delta_j} + \spec{\Delta'_k} \subseteq \{0\} \quad\text{and}\quad \spec{\Delta_j} + \essspec{\Delta'_k} \subseteq \{0\}
 \end{equation}
 for all $j,k\in \mathbb Z$ such that $j + k = i$, so \cref{item:characterization_of_compactness_of_S_for_product_complex_N_compact} $\iff$ \cref{item:characterization_of_compactness_of_S_for_product_complex_trivial}.
 The equivalence \cref{item:characterization_of_compactness_of_S_for_product_complex_N_compact} $\iff$ \cref{item:characterization_of_compactness_of_S_for_product_complex_N_compact_restrictions} is obvious as the Laplacian of the tensor product complex, and hence also $N_i$, respects the decomposition $(H\hgtensor H')_i = \bigoplus_{j+k=i} H_j\htensor H'_k$.
 
 Now assume that $(H,d)$ and $(H',d')$ are nondegenerate.
 If both $j \in \supp(H,d)$ and $k \in \supp(H',d')$, then $\spec{\Delta_j} \not\subseteq\{0\}$ and $\spec{\Delta'_k}\not\subseteq \{0\}$ by \cref{hilbert_complex_nondegeneracy}.
 It is clear that $\essspec{\Delta_j} = \essspec{\Delta'_k} = \emptyset$ for $j$ and $k$ as in \cref{item:characterization_of_compactness_of_S_for_product_complex_nondegenerate} implies \cref{eq:proof_characterization_of_compactness_of_S_for_product_complex_1} for those $j$ and $k$.
 If $H_j$ or $H'_k$ is trivial, then \cref{eq:proof_characterization_of_compactness_of_S_for_product_complex_1} holds since the Laplacian is then the zero operator with empty spectrum.
 This shows \cref{item:characterization_of_compactness_of_S_for_product_complex_nondegenerate} $\implies$ \cref{item:characterization_of_compactness_of_S_for_product_complex_trivial}.
 Conversely, if \cref{item:characterization_of_compactness_of_S_for_product_complex_trivial} holds true, suppose $j \in \supp(H,d)$ and $k \in \supp(H',d')$ with $j+k=i$.
 Then $\spec{\Delta_j} \not\subseteq \{0\}$ and $\spec{\Delta'_k}\not\subseteq \{0\}$ by \cref{hilbert_complex_nondegeneracy} and hence \cref{eq:proof_characterization_of_compactness_of_S_for_product_complex_1} forces $\essspec{\Delta_j} = \essspec{\Delta'_k} = \emptyset$.
 
 The equivalence \cref{item:characterization_of_compactness_of_S_for_product_complex_nondegenerate} $\iff$ \cref{item:characterization_of_compactness_of_S_for_product_complex_nondegenerate_essential_spectrum_empty} is clear from \cref{eq:essential_spectrum_of_product_laplacian} and non-degeneracy.
 We have $\essspec{\Delta_j} = \emptyset$ if and only if $N_j(H,d)$ is compact (so that $\essspec{\Delta_j} \subseteq \{0\}$) and $\dim{\ker(\Delta_j)} = \dim{\cohom[j]{H,d}} < \infty$ (so that $0 \not\in \essspec{\Delta_j}$ by item~\cref{item:N_bounded_zero_in_essential_spectrum} of \cref{inverse_of_laplacian_properties_bounded}), and similarly for $\essspec{\Delta'_k}$.
 This shows \cref{item:characterization_of_compactness_of_S_for_product_complex_nondegenerate} $\iff$ \cref{item:characterization_of_compactness_of_S_for_product_complex_nondegenerate_compactness_of_factors}.
\end{proof}

We now provide several immediate corollaries concerning the non-compactness of $N$ and, by \cref{characterization_of_compactness_of_S}, non-compactness of the minimal solution operators.

\begin{cor}
 Let $(H,d)$ and $(H',d')$ be two nondegenerate Hilbert complexes as in \cref{characterization_of_compactness_of_S_for_product_complex}.
 Assume that there is $j \in \mathbb Z$ such that $N_j(H,d)$ is not compact on $H_j$.
 Then
 $$ N_{j+k}\colon (H\hgtensor H')_{j+k} \to (H\hgtensor H')_{j+k} $$
 is not compact either for all $k \in \supp(H',d')$.
\end{cor}

\begin{proof}
 In this case, $j \in \supp(H,d)$ and $\essspec{\Delta_j}$ is not empty (it contains values other than $0$) by \cref{characterization_of_compactness_of_S}.
 Now apply \cref{characterization_of_compactness_of_S_for_product_complex}.
\end{proof}


\begin{cor}\label{solution_operator_noncompact_product_inf_dim_cohom}
 Let $(H,d)$ and $(H',d')$ be two nondegenerate Hilbert complexes as in \cref{characterization_of_compactness_of_S_for_product_complex}.
 Let $\Delta$ and $\Delta'$ be their respective Laplacians and denote by $\widetilde\Delta$ the Laplacian of the tensor product complex $(H \hgtensor H',d \hgtensor d')$. 
 \begin{enumerate}
  \item If there exists $i \in \mathbb Z$ such that
   $$ \dim{\ker(\widetilde\Delta_i)} = \dim{\cohom[i]{H\hgtensor H',d\hgtensor d'}} = \infty, $$
   then $N_i \colon (H\hgtensor H')_i \to (H\hgtensor H')_i$ is not compact.
  \item If there exists $j\in \mathbb Z$ such that
   $$ \dim{\ker(\Delta_j)} = \dim{\cohom[j]{H,d}} = \infty, $$
   then $N_{j+k}\colon (H\hgtensor H')_{j+k} \to (H\hgtensor H')_{j+k}$ is not compact for all $k\in\supp(H',d')$.
 \end{enumerate}
\end{cor}

\begin{proof}
 In the first case $0 \in \essspec{\widetilde\Delta_i}$, while $j \in \supp(H,d)$ and $0 \in \essspec{\Delta_j}$ in the second case.
 Now apply \cref{characterization_of_compactness_of_S_for_product_complex}.
\end{proof}

\section{Complexes of differential operators}

The main examples of Hilbert complexes are (closed extensions of) complexes of differential operators arising in differential geometry.
By this we mean a sequence of differential operators
$$ 0 \to C^\infty_c(M,E_0) \hilbcomprightarrow{d_0} C^\infty_c(M,E_1) \hilbcomprightarrow{d_1} C^\infty_c(M,E_2) \hilbcomprightarrow{d_2} \cdots \hilbcomprightarrow{d_{n-1}} C^\infty_c(M,E_n) \to 0 $$
with smooth vector bundles $E_i$ over a smooth manifold $M$.
We will denote such a complex simply by $(E,d)$.
Suppose that $M$ is Riemannian and that all $E_i$, $0\leq i \leq n$, are Hermitian bundles, so that we may consider the spaces $L^2(M,E_i)$ of square-integrable measurable sections of $E_i$.
The complex is called \emph{elliptic} if all the Laplacians
$$ \Delta^E_i \coloneqq d_i^t d_i + d_{i-1}d_{i-1}^t \colon C^\infty_c(M,E_i) \to C^\infty_c(M,E_i) $$
are elliptic differential operators, where $d_i^t \colon C^\infty_c(M,E_{i-1}) \to C^\infty_c(M,E_i)$ denotes the formal adjoint of $d_i$.
Any extension of $d_i$ to a closed operator on $L^2(M,E_i)$ must necessarily lie between its closure $d_{i,\minn} \coloneqq \overline{d_i}$ (the \emph{minimal}, or \emph{strong}, extension) and its \emph{maximal} (or \emph{weak}) closed extension $d_{i,\maxx}$, defined as the distributional derivative on
\begin{equation}\label{eq:maximal_extension_of_differential_operator_domain}
 \dom{d_{i,\maxx}} \coloneqq \big\{ s \in L^2(M,E_i) : d_i s \in L^2(M,E_i) \text{ in the sense of distributions}\big\}.
\end{equation}
A choice of closed extensions for $d_i$ that produces a Hilbert complex is also called an \emph{ideal boundary condition}.
Such ideal boundary conditions always exist.
Indeed, the minimal and maximal extensions themselves give rise to ideal boundary conditions,
$$ (L^2(M,E_\bullet),d_\minn) \quad\text{and}\quad (L^2(M,E_\bullet),d_\maxx), $$
see \cite[Lemma~3.1]{Bruening1992}.
We will mostly deal with the maximal extensions in this article.

Now consider two complexes of differential operators, say $(E,d^E)$ and $(F,d^F)$ over manifolds $M$ and $N$, respectively.
We proceed similarly to the construction of the tensor product of Hilbert complexes in order to obtain a complex of differential operators on $M \times N$.
Set
$$ (E\getensor F)_i \coloneqq \bigoplus_{j+k=i} E_j \boxtimes F_k, $$
where $E_j \boxtimes F_k \coloneqq (\pi_M^* E_j) \otimes (\pi_N^* F_k)$, with $\pi_M \colon M\times N \to M$ and $\pi_N \colon M\times N \to N$ the projections, is a vector bundle over $M\times N$ with fiber $(E_j)_x \otimes (F_k)_y$ over $(x,y) \in M\times N$.
If $M$ and $N$ are Riemannian and all vector bundles are Hermitian, then $M \times N$ and $(E\getensor F)_i$ are also equipped with metrics in a canonical way.

By $C^\infty_c(M,E_\bullet)$ we denote the $\mathbb Z$-graded vector space $\bigoplus_j C^\infty_c(M,E_j)$, and similarly for $C^\infty_c(N,F_\bullet)$.
Their graded tensor product $C^\infty_c(M,E_\bullet) \gtensor C^\infty_c(N,F_\bullet)$ is then defined as in \cref{eq:graded_vector_space_tensor_product}.
The following \namecref{unique_differential_operator_on_product_bundle} can be found in \cite[p.~110]{Bruening1992}:

\begin{lem}\label{unique_differential_operator_on_product_bundle}
 If $(E,d^E)$ and $(F,d^F)$ are complexes of differential operators, then there exists a unique complex of differential operators
 $$ d^{E\getensor F}_i \colon C^\infty_c(M\times N,(E\getensor F)_i) \to C^\infty_c(M\times N,(E\getensor F)_{i+1}) $$
 such that the diagram
 \begin{center}
  \begin{tikzcd}
   \cdots
    \arrow{r}{d^E\gtensor d^F}
   &
   (C^\infty_c(M,E_\bullet) \gtensor C^\infty_c(N,F_\bullet))_i
    \arrow{r}{d^E\gtensor d^F}
    \arrow{d}{\iota_i}
   &
   (C^\infty_c(M,E_\bullet) \gtensor C^\infty_c(N,F_\bullet))_{i+1}
    \arrow{r}{d^E\gtensor d^F}
    \arrow{d}{\iota_{i+1}}
   &
   \cdots
   
   \\
   
   \cdots
    \arrow{r}{d^{E\getensor F}}
   &
   C^\infty_c(M\times N,(E\getensor F)_i)
    \arrow{r}{d^{E\getensor F}}
   &
   C^\infty_c(M\times N,(E\getensor F)_{i+1})
    \arrow{r}{d^{E\getensor F}}
   &
   \cdots
  \end{tikzcd}
 \end{center}
 commutes, where $d^E \gtensor d^F$ is given by
 \begin{multline}\label{eq:unique_differential_operator_on_product_bundle_def_of_operator}
  \bigoplus_{j+k=i} \big(d^E_j \otimes \id{C^\infty_c(N,F_k)} + \sigma_j \otimes d^F_k\big) \colon \\
  (C^\infty_c(M,E_\bullet) \gtensor C^\infty_c(N,F_\bullet))_i \to (C^\infty_c(M,E_\bullet) \gtensor C^\infty_c(N,F_\bullet))_{i+1},
 \end{multline}
 with $\sigma_j \colon C^\infty_c(M,E_j) \to C^\infty_c(M,E_j)$ the multiplication by $(-1)^j$, and
 $$ \iota_i \colon (C^\infty_c(M,E_\bullet) \gtensor C^\infty_c(N,F_\bullet))_i \to C^\infty_c(M\times N,(E\getensor F)_i) $$
 is the canonical inclusion given by $\iota_i(s \otimes t)(x,y) \coloneqq s(x) \otimes t(y)$ for $s \in C^\infty_c(M,E_j)$ and $t \in C^\infty_c(N,F_k)$.
 If $(E,d^E)$ and $(F,d^F)$ are elliptic complexes, then so is $(E\getensor F, d^{E\getensor F})$.
\end{lem}

In the proof of \cref{unique_differential_operator_on_product_bundle}, one uses the fact that, via $\iota_i$, the space $(C^\infty_c(M,E_\bullet) \gtensor C^\infty_c(N,F_\bullet))_i$ is sequentially dense in $C^\infty_c(M\times N,(E\getensor F)_i)$ for the $LF$-topology.

\begin{example}\label{ex:de_rham_complex}
 Let $M$ and $N$ be smooth manifolds, and consider their de~Rham complexes
 $$ d^M_j \colon \Omega^j_c(M) \to \Omega^{j+1}_c(M) \quad\text{and}\quad d^N_k \colon \Omega^k_c(N) \to \Omega^{k+1}_c(N), $$
 where $\Omega^j_c(M) \coloneqq C^\infty_c(M,\Lambda^j T^*M)$ and similarly for $\Omega^k_c(N)$, so that $E_j = \Lambda^j T^*M$ and $F_k = \Lambda^k T^*N$ in the language of \cref{unique_differential_operator_on_product_bundle}.
 Since the cotangent bundle of the product $M\times N$ splits as $T^*(M \times N) \cong \pi_M^*(T^*M) \oplus \pi_N^*(T^*N)$, we get
 \begin{equation}\label{eq:exterior_power_of_cotangent_bundle_of_product_manifold}
  \Lambda^i T^*(M\times N) \cong \bigoplus_{j+k=i} \pi_M^*(\Lambda^j T^*M) \otimes \pi_N^*(\Lambda^k T^*N) = \bigoplus_{j+k=i} (\Lambda^j T^*M)\boxtimes (\Lambda^k T^*N)
 \end{equation}
 from the properties of the exterior algebra functor, hence $(E\getensor F)_i$ is the vector bundle of $i$-forms on $M \times N$, and
 $$ d^{E\getensor F}_i \colon C^\infty_c(M\times N,\Lambda^i T^*(M\times N)) \to C^\infty_c(M\times N,\Lambda^{i+1} T^*(M\times N)) $$
 is the de~Rham differential for the product manifold, since this obviously extends \cref{eq:unique_differential_operator_on_product_bundle_def_of_operator} by the Leibniz rule for the exterior derivative.
 Note that when accounting for the isomorphism \cref{eq:exterior_power_of_cotangent_bundle_of_product_manifold}, the map $\iota_i \colon \bigoplus_{j+k=i} \Omega^j_c(M) \otimes \Omega^k_c(N) \to \Omega^i_c(M\times N)$
 is given by $\iota_i(\omega \otimes \eta) = \pi_M^*\omega \wedge \pi_N^*\eta$.
\end{example}

\begin{example}\label{ex:dolbeault_complex}
 Let $X$ and $Y$ be complex manifolds, $E \to X$ and $F \to Y$ two holomorphic vector bundles, and consider, for fixed $1\leq p'\leq \dim[\mathbb C]{X}$ and $1 \leq p'' \leq \dim[\mathbb C]{Y}$, the Dolbeault complexes
 $$ \dbar^E_{p',\bullet} \colon \Omega^{p',\bullet}_c(X,E) \to \Omega^{p',\bullet+1}_c(X,E) \quad\text{and}\quad \dbar^F_{p'',\bullet} \colon \Omega^{p'',\bullet}_c(Y,F) \to \Omega^{p'',\bullet+1}_c(Y,F), $$
 where $\Omega^{p',q'}_c(X,E) \coloneqq C^\infty_c(X,\Lambda^{p',q'}T^*X \otimes E)$ denotes the space of compactly supported smooth $(p',q')$ forms on $X$ with values in $E$.
 One might expect the resulting tensor product complex on $X\times Y$ to be the $\dbar^{E\boxtimes F}$-complex, with $E\boxtimes F \coloneqq \pi_X^*E \otimes \pi_Y^*F$, restricted to those $(p'+p'',q)$ forms 
 which are sections of
 \begin{equation}\label{eq:dolbeault_complex_bundle_partial_forms}
  \pi_X^*(\Lambda^{p',0}T^*X) \otimes \pi_Y^*(\Lambda^{p'',0}T^*Y) \otimes \Lambda^{0,\bullet}T^*(X \times Y) \otimes (E\boxtimes F).
 \end{equation}
 This is true up to a sign factor.
 Consider the cochain complex
 $$ \dbar^E_{p',\bullet} \getensor (-1)^{p'}\dbar^F_{p'',\bullet} \colon \Omega^{p',\bullet}_c(X,E) \gtensor \Omega^{p'',\bullet}_c(X,F) \to \Omega^{p',\bullet}_c(X,E) \gtensor \Omega^{p'',\bullet}_c(X,F) $$
 as in \cref{eq:unique_differential_operator_on_product_bundle_def_of_operator}, and the dense inclusions (for the $LF$-topology)
 \begin{equation}\label{eq:dolbeault_complex_embedding}
  \iota_q^{p',p''} \colon (\Omega^{p',\bullet}_c(X,E) \gtensor \Omega^{p'',\bullet}_c(Y,F))_q \to \bigoplus_{q'+q''=q} C^\infty_c(X\times Y,(\Lambda^{p',q'}T^*X \otimes E) \boxtimes (\Lambda^{p'',q''}T^*Y \otimes F))
 \end{equation}
 given, as in \cref{unique_differential_operator_on_product_bundle}, by $\iota_q^{p',q'}(\omega \otimes \eta) \coloneqq \pi_X^* \omega \otimes \pi_Y^*\eta$.
 We denote the right hand side of \cref{eq:dolbeault_complex_embedding} by $\Omega_c(E,F)_q^{p',p''}$.
 Note that this may be identified with the space of smooth compactly supported sections of \cref{eq:dolbeault_complex_bundle_partial_forms}.
 According to the bundle isomorphism
 $$ \Lambda^{p,q}T^*(X\times Y)\otimes (E\boxtimes F) \cong \bigoplus_{\substack{p'+p''=p \\ q'+q''=q}} (\Lambda^{p',q'}T^*X \otimes E)\boxtimes (\Lambda^{p'',q''}T^*Y\otimes F), $$
 the full space of $(p,q)$ forms decomposes as $\Omega^{p,q}_c(X\times Y,E\boxtimes F) \cong \bigoplus_{p'+p''=p} \Omega_c(E,F)_q^{p',p''}$.
 Now for $\omega \in \Omega^{p',q'}_c(X,E)$ and $\eta \in \Omega^{p'',q''}_c(Y,F)$ with $p'+p''=p$ and $q'+q''=q$, we have $\iota_q^{p',p''}(\omega\otimes\eta) \in \Omega_c(E,F)_q^{p',p''}$ and, with $\dbar^{E\boxtimes F}$ being understood as up to the above isomorphism,
 $$ \dbar^{E\boxtimes F}_{p,q}(\iota_q^{p',p''}(\omega\otimes\eta)) = \pi_X^*\big(\dbar^E_{p',q'}\omega\big)\otimes \pi_Y^*\eta + (-1)^{q'}\pi_X^*\omega \otimes \pi_Y^*\big((-1)^{p'}\dbar^F_{p'',q''}\eta\big) \in \Omega_c(E,F)_{q+1}^{p',p''} $$
 because the total degree of $\omega$ is $p'+q'$, and this is precisely $\iota_q^{p',p''}(\dbar^E_{p',\bullet} \getensor (-1)^{p'}\dbar^F_{p'',\bullet})(\omega \otimes \eta)$.
 By \cref{unique_differential_operator_on_product_bundle}, the restriction of $\dbar^{E\boxtimes F}$ to  $\Omega_c(E,F)_\bullet^{p',p''}$ is the unique complex of differential operators extending $\dbar^E_{p',\bullet} \getensor (-1)^{p'}\dbar^F_{p'',\bullet}$ via $\iota_\bullet^{p',p''}$.
 Note that the situation is a lot simpler (as simple as in \cref{ex:de_rham_complex}) if one only considers $(0,q)$ forms.
\end{example}

We now extend the above situation to the level of Hilbert complexes obtained from $(E,d^E)$ and $(F,d^F)$.
First note that the inclusions $\iota_i$ extend to a unitary isomorphism of graded Hilbert spaces
$$ \hat\iota \coloneqq \bigoplus_i \hat\iota_i \colon L^2(M,E_\bullet) \hgtensor L^2(N,F_\bullet) \xrightarrow{\cong} L^2(M\times N,(E\getensor F)_\bullet), $$
where $L^2(M,E_\bullet) \coloneqq \bigoplus_j L^2(M,E_j)$, and similarly for $L^2(N,F_\bullet)$ and $L^2(M\times N,(E\getensor F)_\bullet)$.
The next result is Lemma~3.6 in \cite{Bruening1992}:

\begin{lem}\label{maximal_extensions_of_differential_complexes_product}
 Let $(E,d^E)$ and $(F,d^F)$ be complexes of differential operators with Hermitian bundles over Riemannian manifolds, and $(E\getensor F,d^{E\gtensor F})$ their tensor product as in \cref{unique_differential_operator_on_product_bundle}.
 Then the diagram 
 \begin{center}
  \begin{tikzcd}
   \cdots
    \arrow{r}{d^E_\maxx \hgtensor d^F_\maxx}
   &
   \dom{(d^E_\maxx \hgtensor d^F_\maxx)_i}
    \arrow{r}{d^E_\maxx \hgtensor d^F_\maxx}
    \arrow{d}{\hat\iota_i}[swap]{\cong}
   &
   \dom{(d^E_\maxx \hgtensor d^F_\maxx)_{i+1}}
    \arrow{r}{d^E_\maxx \hgtensor d^F_\maxx}
    \arrow{d}{\hat\iota_{i+1}}[swap]\cong
   &
   \cdots
   
   \\
   
   \cdots
    \arrow{r}{d^{E\getensor F}_\maxx}
   &
   \dom{d^{E\getensor F}_{i,\maxx}}
    \arrow{r}{d^{E\getensor F}_\maxx}
   &
   \dom{d^{E\getensor F}_{i+1,\maxx}}
    \arrow{r}{d^{E\getensor F}_\maxx}
   &
   \cdots
  \end{tikzcd}
 \end{center}
 commutes, where $d^E_\maxx$, $d^F_\maxx$ and $d^{E\getensor F}_\maxx$ denote the (differentials of the) Hilbert complexes of the maximal closed extensions of $d^E$, $d^F$ and $d^{E\getensor F}$, respectively, and $d^E_\maxx \hgtensor d^F_\maxx$ is the differential of the tensor product Hilbert complex, see \cref{def:product_hilbert_complex}.
 In other words, $\hat\iota$ is a unitary equivalence between $(L^2(M,E_\bullet),d^E_\maxx)\hgtensor (L^2(N,F_\bullet),d^F_\maxx)$ and $(L^2(M\times N,(E\getensor F)_\bullet),d^{E\getensor F}_\maxx)$.
 An analogous statement holds for the minimal extensions.
\end{lem}

In particular, \cref{maximal_extensions_of_differential_complexes_product} implies that the \emph{Gaffney extension} of the $d^{E\getensor F}$-Laplacian, which is the Laplacian of the Hilbert complex $(L^2(M\times N,(E\getensor F)_\bullet),d^{E\getensor F}_w)$ and is defined by
$$ \Delta^{E\getensor F}_G \coloneqq (d^{E\getensor F}_\maxx)^* d^{E\getensor F}_\maxx + d^{E\getensor F}_\maxx (d^{E\getensor F}_\maxx)^* $$
with domain
\begin{multline*}
 \dom{\Delta^{E\getensor F}_G} \coloneqq \big\{x \in \dom{d^{E\getensor F}_\maxx} \cap \dom{(d^{E\getensor F}_\maxx)^*} : \\ d^{E\getensor F}_\maxx x \in \dom{(d^{E\getensor F}_\maxx)^*} \text{ and } (d^{E\getensor F}_\maxx)^*x \in \dom{d^{E\getensor F}_\maxx} \big\},
\end{multline*}
is a self-adjoint extension of the $d^{E\getensor F}$-Laplacian on $L^2(M\times N,(E\getensor F)_\bullet)$ that is unitarily equivalent to the Laplacian of the Hilbert complex $(L^2(M,E_\bullet) \hgtensor L^2(N,F_\bullet), d^E_\maxx \hgtensor d^F_\maxx)$.
As a consequence, the two Laplacians share all of their spectral and operator theoretic properties.

\section{Applications to the \texorpdfstring{$\dbar$}{dbar}-complex}\label{sec:dbar}

We will now apply the general theory developed in the previous sections to the $\dbar$-Neumann problem.
For a Hermitian manifold $X$ and a Hermitian holomorphic vector bundle $E$ over $X$, we consider, for fixed $1 \leq p \leq \dim[\mathbb C]{X}$, the complex of differential operators%
\footnote{We will usually omit the reference to the bidegree in $\dbar^E_{p,q}$.}
\begin{equation}\label{eq:dbar_operator}
 \dbar^E_{p,q} \colon \Omega^{p,q}_c(X,E) \to \Omega^{p,q+1}_c(X,E)
\end{equation}
and its Laplacians
$$ \square^E_{p,q} \coloneqq \dbar^{E,t}\dbar^E + \dbar^E\dbar^{E,t} \colon \Omega^{p,q}_c(X,E) \to \Omega^{p,q}_c(X,E), $$
where $\dbar^{E,t}$ is the formal adjoint to $\dbar^E$.
Details on the definition and properties of $\dbar^E$ can be found, for instance, in \cite{Wells1973,Ma2007,Huybrechts2005,Demailly2012,Demailly2013}.
For extensive surveys of the $L^2$ theory of $\dbar$, with a focus on pseudoconvex domains in $\mathbb C^n$, see \cite{Straube2010,Chen2001}.
The \emph{Gaffney extension} of $\square^E_{p,q}$, which we still denote by $\square^E_{p,q}$, is given by
$$ \square^E_{p,q} \coloneqq \dbar^{E,*} \dbar^E + \dbar^E \dbar^{E,*} \colon \dom{\square^E_{p,q}} \subseteq L^2_{p,q}(X,E) \to L^2_{p,q}(X,E) $$
with domain
$$ \dom{\square^E_{p,q}} \coloneqq \big\{\omega \in \dom{\dbar^E} \cap \dom{\dbar^{E,*}} : \dbar^E \omega \in \dom{\dbar^{E,*}} \text{ and } \dbar^{E,*}\omega \in \dom{\dbar^E}\big\}, $$
where $\dbar^E$ is now understood as the maximal extension of \cref{eq:dbar_operator} to a closed operator from $L^2_{p,q}(X,E)$ to $L^2_{p,q+1}(X,E)$, see \cref{eq:maximal_extension_of_differential_operator_domain}, and $\dbar^{E,*}$ is its Hilbert space adjoint.
Here, $L^2_{p,q}(X,E) \coloneqq L^2(X,\Lambda^{p,q}T^*X \otimes E)$ denotes the space of square-integrable $(p,q)$ forms on $X$ with values in $E$.
In this way, we obtain a Hilbert complex $(L^2_{p,\bullet}(X,E),\dbar^E_{p,\bullet})$ with Laplacian $\square^E_{p,\bullet}$ for every $1 \leq p \leq \dim[\mathbb C]{X}$.

The inverse of $\square^E$, in the sense of \cref{inverse_of_laplacian_properties}, is customarily called the \emph{$\dbar$-Neumann operator} and denoted by $N^E$.
We denote by $N^E_{p,q}$ and $S^E_{p,q}$ the restrictions of $N^E$ and $S^E$, respectively, to $L^2_{p,q}(X,E)$.
By \cref{equivalent_statements_to_N_bounded}, $N_{p,q}$ is bounded if and only if $\dbar^E$ on both $(p,q-1)$ and $(p,q)$ forms has closed range.
In this case, the minimal (or canonical) solution operator $S^E$ to the $\dbar^E$-equation is also bounded on $L^2_{p,q}(X,E)$ and on $L^2_{p,q+1}(X,E)$, and we have
$$ S^E = \dbar^{E,*} N^E $$
on $L^2_{p,q}(X,E)$ by \cref{inverse_of_laplacian_properties_bounded}.
The cohomology of the Hilbert complex $(L^2_{p,\bullet}(X,E),\dbar^E)$ is the \emph{$L^2$-Dolbeault cohomology},
$$ \ltwocohom[p,q]{X,E} \coloneqq \cohom[q]{L^2_{p,\bullet}(X,E),\dbar^E} = \ker(\dbar^E) \cap L^2_{p,q}(X,E)\mathbin{\big/}\img{\dbar^E}\cap L^2_{p,q}(X,E), $$
and its reduced cohomology is the \emph{reduced $L^2$-Dolbeault cohomology},
$$ \redltwocohom[p,q]{X,E} \coloneqq \redcohom[q]{L^2_{p,\bullet}(X,E),\dbar^E} = \ker(\dbar^E) \cap L^2_{p,q}(X,E)\mathbin{\big/} \overline{\img{\dbar^E}\cap L^2_{p,q}(X,E)}, $$
which is canonically isomorphic to $\ker(\square^E_{p,q})$.
For instance,
\begin{equation}\label{eq:bergman_space_cohomology}
 \bergman(X,E) \coloneqq \ker(\dbar^E) \cap L^2(X,E) = \ker(\dbar^{E,*}\dbar^E) \cap L^2(X,E) \cong \redltwocohom[0,0]{X,E}
\end{equation}
is the space of square-integrable holomorphic sections of $E$, called the \emph{Bergman space} of $E \to X$.
Of course, both cohomology spaces coincide if $\dbar^E$ has closed range.
Our main result for this section is the following:

\begin{thm}\label{essential_spectrum_of_product_dbar}
 Let $E \to X$ and $F \to Y$ be Hermitian holomorphic vector bundles over Hermitian manifolds.
 Then, for $0 \leq p,q \leq \dim[\mathbb C]{X} + \dim[\mathbb C]{Y}$,
 \begin{equation}\label{eq:spectrum_of_product_dbar}
  \spec{\square^{E\boxtimes F}_{p,q}} = \bigcup_{\substack{p'+p''=p \\ q'+q''=q}} \big(\spec{\square^E_{p',q'}} + \spec{\square^F_{p'',q''}}\big)
 \end{equation}
 and
 \begin{equation}\label{eq:essential_spectrum_of_product_dbar}
  \essspec{\square^{E\boxtimes F}_{p,q}} = \bigcup_{\substack{p'+p''=p \\ q'+q''=q}} \big(\essspec{\square^E_{p',q'}} + \spec{\square^F_{p'',q''}}\big)\cup\big(\spec{\square^E_{p',q'}} + \essspec{\square^F_{p'',q''}}\big),
 \end{equation}
 where $p'$ and $q'$ range over $\{0,\dotsc,\dim[\mathbb C]{X}\}$, and $p''$ and $q''$ range over $\{0,\dotsc,\dim[\mathbb C]{Y}\}$.
\end{thm}

\begin{proof}
 Fix $p'$ and $p''$ for the moment and denote by $L^2(E,F)_q^{p',p''}$ the completion of $\Omega_c(E,F)_q^{p',p''}$ as in \cref{ex:dolbeault_complex}, with respect to the induced Hermitian structures.
 Consider the Hilbert complex $(L^2(E,F)_\bullet^{p',p''},\dbar^{E\boxtimes F})$, obtained by taking the maximal extension of $\dbar^{E\boxtimes F}$, restricted to $\Omega_c(E,F)_\bullet^{p',p''}$, to a closed operator on $L^2(E,F)_q^{p',p''}$.
 By \cref{unique_differential_operator_on_product_bundle,ex:dolbeault_complex}, we know that this Hilbert complex is unitarily equivalent to
 $$ \big(L^2_{p',\bullet}(X,E)\hgtensor L^2_{p'',\bullet}(Y,F),\dbar^E_{p',\bullet}\hgtensor (-1)^{p'}\dbar^F_{p'',\bullet}\big), $$
 which is the tensor product of $(L^2_{p',\bullet}(X,E),\dbar^E_{p',\bullet})$ and $(L^2_{p'',\bullet}(Y,F),(-1)^{p'}\dbar^F_{p'',\bullet})$, as in \cref{def:product_hilbert_complex}.
 Now for $0 \leq p,q \leq \dim[\mathbb C]{X} + \dim[\mathbb C]{Y}$, we have
 $$ L^2_{p,q}(X\times Y,E\boxtimes F) \cong \bigoplus_{p'+p''=p} L^2(E,F)_q^{p',p''}, $$
 which is due to the fact that $X\times Y$ is Hermitian and hence forms with different bidegree (but same total degree) are orthogonal.
 It follows that $(L^2_{p,\bullet}(X\times Y,E\boxtimes F),\dbar^{E\boxtimes F})$ is unitarily equivalent to the direct sum of Hilbert complexes
 \begin{equation}\label{eq:direct_sum_of_dbar_hilbert_complexes_bidegree}
  \bigoplus_{p'+p''=p} \big(L^2_{p',\bullet}(X,E), \dbar^E_{p',\bullet}\big) \hgtensor \big(L^2_{p'',\bullet}(Y,F), (-1)^{p'}\dbar^F_{p'',\bullet}\big).
 \end{equation}
 Equations \cref{eq:spectrum_of_product_dbar,eq:essential_spectrum_of_product_dbar} now follow immediately from \cref{eq:spectrum_of_product_laplacian,eq:essential_spectrum_of_product_laplacian,eq:direct_sum_of_dbar_hilbert_complexes_bidegree}.
 Note that the Laplacians of $(L^2_{p'',\bullet}(Y,F), (-1)^{p'}\dbar^F_{p'',\bullet})$ and $(L^2_{p'',\bullet}(Y,F), \dbar^F_{p'',\bullet})$ coincide and are equal to $\square^F_{p'',\bullet}$.
\end{proof}

Since the $\dbar$-complex is nondegenerate in the sense of \cref{def:nondegenerate_hilbert_complex}, we obtain the following characterization of compactness of the $\dbar$-Neumann operator from \cref{characterization_of_compactness_of_S_for_product_complex}:

\begin{thm}\label{dbar_neumann_compactness_product}
 Let $E \to X$ and $F \to Y$ be Hermitian holomorphic vector bundles over Hermitian manifolds such that $\dbar^E$ and $\dbar^F$ have closed range (in all bidegrees).
 Then for $0 \leq p,q \leq \dim[\mathbb C]{X} + \dim[\mathbb C]{Y}$, the following are equivalent:
 \begin{enumerate}
  \item The $\dbar$-Neumann operator
   $N_{p,q}^{E\boxtimes F} \colon L^2_{p,q}(X\times Y,E\boxtimes F) \to L^2_{p,q}(X\times Y,E\boxtimes F)$
   is compact.
  \item $\essspec{\square^{E\boxtimes F}_{p,q}} = \emptyset$.
  \item $\essspec{\square^E_{p',q'}} = \essspec{\square^F_{p'',q''}} = \emptyset$ for all $0 \leq p',q' \leq \dim[\mathbb C]{X}$ and $0\leq p'',q'' \leq \dim[\mathbb C]{Y}$ with $p'+p''=p$ and $q'+q''=q$.
  \item For all $0 \leq p',q' \leq \dim[\mathbb C]{X}$ and $0\leq p'',q'' \leq \dim[\mathbb C]{Y}$ with $p'+p''=p$ and $q'+q''=q$, the $L^2$-Dolbeault cohomology spaces
   $$ \ltwocohom[p',q']{X,E} \quad\text{and}\quad \ltwocohom[p'',q'']{Y,F} $$
   have finite dimension and the $\dbar$-Neumann operators
   $$ N^E_{p',q'} \colon L^2_{p',q'}(X,E) \to L^2_{p',q'}(X,E) \quad\text{and}\quad N^F_{p'',q''} \colon L^2_{p'',q''}(Y,F) \to L^2_{p'',q''}(Y,F) $$
   are compact.
 \end{enumerate}
\end{thm}

\begin{cor}\label{bergman_space_infinite_means_noncompactness}
 Let $E \to X$ and $F \to Y$ be Hermitian holomorphic vector bundles over complex manifolds such that $\dbar^E$ and $\dbar^F$ have closed range (in all bidegrees).
 \begin{enumerate}
  \item
   If the $L^2$-Dolbeault cohomology space $\ltwocohom[p,q]{X\times Y,E\boxtimes F}$ has infinite dimension, then
   $$ N^{E\boxtimes F}_{p,q} \colon L^2_{p,q}(X\times Y,E\boxtimes F) \to L^2_{p,q}(X\times Y,E\boxtimes F) $$
   is not compact.
  \item 
   If either of the Bergman spaces
   $$ \bergman(X,E) = L^2(X,E) \cap \hol{X,E} \quad\text{or}\quad \bergman(Y,F) = L^2(Y,F) \cap \hol{Y,F} $$
   of holomorphic $L^2$ sections of $E$, respectively $F$, has infinite dimension, then $N^{E\boxtimes F}_{p,q}$
   is not compact for all $0 \leq p,q \leq \dim[\mathbb C]{Y}$, respectively $0 \leq p,q \leq \dim[\mathbb C]{X}$.
 \end{enumerate}
\end{cor}

\begin{proof}
 This is immediate from \cref{solution_operator_noncompact_product_inf_dim_cohom} by using $\bergman(X,E) = \ker(\square^E_0) \cong \ltwocohom[0,0]{X,E}$ as in \cref{eq:bergman_space_cohomology}.
\end{proof}

\begin{rem}
 \begin{inlineenum}
  \item
  We can use higher degree $L^2$-Dolbeault cohomology spaces of one factor instead of the Bergman spaces as in \cref{bergman_space_infinite_means_noncompactness} to conclude non-compactness, see \cref{solution_operator_noncompact_product_inf_dim_cohom}.
  
  \item
  The above results also apply when replacing $\dbar^E$ by the minimal (or strong) extensions (i.e., the closure) of $\dbar^E \colon \Omega^{p,q}_c(X,E) \to \Omega^{p,q+1}_c(X,E)$, and similarly for $\dbar^F$.
  This follows immediately from the fact that \cref{maximal_extensions_of_differential_complexes_product} also holds for the minimal extensions of differential operators.
 \end{inlineenum}
\end{rem}

 \begin{rem}
   Regarding the closed range property for $\dbar^E$, one has the following sufficient condition from \cite[Theorem~3.1.8]{Ma2007}:
   Assume there is a compact subset $K \subseteq X$ and $C > 0$ such that
   \begin{equation}\label{eq:dbar_closed_range_fundamental_estimate}
    \|u\|^2 \leq C\big( \|\dbar^E u\|^2 + \|\dbar^{E,*}u\|^2\big) + \int_K |u|^2\,dv_X
   \end{equation}
   is satisfied for all $u \in \dom{\dbar^E}\cap\dom{\dbar^{E,*}}\cap L^2_{p,q}(X,E)$, where $v_X$ denotes the measure induced by the Riemannian volume form on $X$.
   Then $\dbar^E_{p,q}$ has closed range and the $L^2$-Dolbeault cohomology space $\ltwocohom[p,q]{X,E}$ has finite dimension.
   One situation where this is satisfied is when the given metric on $X$ is complete and the self-adjoint bundle endomorphism
   $$ A_{E,\omega} \coloneqq [i\Theta(E),\Lambda] + T_\omega \colon \Lambda^{p,q}T^*X \otimes E \to \Lambda^{p,q}T^*X \otimes E $$
   satisfies $A_{E,\omega} \geq c$ for some $c > 0$ at infinity, i.e., outside some compact subset of $X$.
   Here, $\Theta(E)$ is the curvature of $E$ for the Chern connection, $\omega$ denotes the $(1,1)$-form associated to the metric, $\Lambda$ is the adjoint of exterior multiplication by $\omega$, and $T_\omega$ is a torsion term which vanishes if $X$ is K\"ahler.
   Indeed, in this case, the \emph{Bochner-Kodaira-Nakano} inequality
   $$ \|\dbar^E u\|^2 + \|\dbar^{E,*}u\|^2 \geq \int_X \langle A_{E,\omega}u,u \rangle\,dv_X $$
   extends to all $u \in \dom{\dbar^E}\cap\dom{\dbar^{E,*}}\cap L^2_{p,q}(X,E)$ since $C^\infty_c(X,\Lambda^{p,q}T^*X\otimes E)$ is dense in the former space for the norm $u \mapsto \|u\|^2 + \|\dbar^E u\|^2 + \|\dbar^{E,*}u\|^2$ by completeness, see \cite{Andreotti1965} or \cite[Lemma~3.3.1]{Ma2007}.
   Using the positivity condition on $A_{E,\omega}$ at infinity, one easily arrives at \eqref{eq:dbar_closed_range_fundamental_estimate}.
   For more information on $A_{E,\omega}$ and the Bochner-Kodaira-Nakano identity, consult \cite{Demailly1986,Demailly2012,Bertin2002}.
 \end{rem}

The statement of \cref{essential_spectrum_of_product_dbar} can readily be generalized to the product of a finite number of manifolds and vector bundles.
We will conclude this section by considering the situation of several one-dimensional factors.
For simplicity, we will only treat $(0,q)$ forms, and we abbreviate
$$ \square^E_q \coloneqq \square^E_{0,q}, \quad S^E_q \coloneqq S^E_{0,q}, \quad\text{and}\quad N^E_q \coloneqq N^E_{0,q}. $$

\begin{thm}\label{dbar_neumann_noncompactness_riemann_surface_products}
 Let $X_j$, $1 \leq j \leq n$ with $n \geq 2$ be Hermitian Riemann surfaces (i.e., one-dimensional complex manifolds) and $E_j \to X_j$ Hermitian holomorphic vector bundles, such that $\dbar^{E_j}$ has closed range for all $j$.
 Put
 $$ X \coloneqq X_1 \times \dotsb \times X_n \quad\text{and}\quad E \coloneqq \pi_1^* E_1 \otimes \dotsb \otimes \pi_n^* E_n, $$
 with $\pi_j \colon X \to X_j$ the projections.
 \begin{enumerate}
  \item\label{item:dbar_neumann_noncompactness_riemann_surface_products_compactness_of_bot}
   The operator $N^E_0$ is compact if and only if, for all $1 \leq j \leq n$, the minimal solution operator $S^{E_j}_1$ is compact and $\dim{\bergman(X_j,E_j)} < \infty$.
  \item\label{item:dbar_neumann_noncompactness_riemann_surface_products_compactness_of_top}
   The operator $N^E_n$ is compact if and only if, for all $1 \leq j \leq n$, the minimal solution operator $S^{E_j}_1$ is compact and $\dim{\ltwocohom[0,1]{X_j,E_j}} < \infty$.
  \item\label{item:dbar_neumann_noncompactness_riemann_surface_products_compactness_in_middle}
   The operator $N^E_q$ with $q \in \{1,\dotsc,n-1\}$ is compact if and only if both $N^E_0$ and $N^E_n$ are compact.
   (Equivalently: $S^{E_j}_1$ is compact for all $1 \leq j \leq n$ and all factors have finite dimensional $L^2$-Dolbeault cohomology.)
  \item\label{item:dbar_neumann_noncompactness_riemann_surface_products_noncompactness_of_bot}
   If $N^E_0$ is not compact, then $N^E_q$ is also not compact for $q \in \{0,\dotsc,n-1\}$.
  \item\label{item:dbar_neumann_noncompactness_riemann_surface_products_noncompactness_of_top}
   If $N^E_n$ is not compact, then $N^E_q$ is also not compact for $q \in \{1,\dotsc,n\}$.
  \item\label{item:dbar_neumann_noncompactness_riemann_surface_products_compactness_boils_up}
   If $N^E_{q_0}$ is not compact for some $q_0 \in \{1,\dotsc,n-1\}$, then $N^E_q$ is also not compact for all $q \in \{1,\dotsc,n-1\}$.
  \item\label{item:dbar_neumann_noncompactness_riemann_surface_products_noncompactness_of_S}
   If $S^{E_j}_1$ is not compact for some $1 \leq j \leq n$, then $N^E_q$ is not compact for all $q \in \{0,\dotsc,n\}$.
  \item\label{item:dbar_neumann_noncompactness_riemann_surface_products_bergman_space_one_factor}
   If there exists $j_0 \in \{1,\dotsc,n\}$ such that the Bergman space $\bergman(X_{j_0},E_{j_0})$ has infinite dimension, then $N^E_q$ is not compact for all $0 \leq q \leq n-1$.
 \end{enumerate}
\end{thm}

\begin{proof}[Proof of \cref{dbar_neumann_noncompactness_riemann_surface_products}]
 The appropriate formula for the essential spectrum of $\square^E_i$ in the case of several factors is
 \begin{equation}\label{eq:dbar_neumann_noncompactness_riemann_surface_products_essential_spectrum}
  \essspec{\square^E_q} = \bigcup_{\substack{K \in \{0,1\}^n \\ \sum_{j=1}^n K_j = q}} \bigcup_{j=1}^n \bigg(\essspec{\square^{E_j}_{K_j}} + \sum_{j' \neq j} \spec{\square^{E_{j'}}_{K_{j'}}}\bigg),
 \end{equation}
 and compactness of $N^E_q$ is equivalent to $\essspec{\square^E_q} = \emptyset$ by item~\cref{item:characterization_of_compactness_of_S_for_product_complex_nondegenerate_essential_spectrum_empty} of \cref{characterization_of_compactness_of_S_for_product_complex}.
 Concerning \cref{item:dbar_neumann_noncompactness_riemann_surface_products_compactness_of_bot}, we have
 $$ \essspec{\square^E_0} = \bigcup_{j=1}^n \bigg(\essspec{\square^{E_j}_0} + \sum_{j' \neq j} \spec{\square^{E_{j'}}_0}\bigg), $$
 hence $\essspec{\square^E_0} \subseteq \{0\}$ if and only if $\essspec{\square^{E_j}_0} = \emptyset$ for all $1 \leq j \leq n$.
 This is the case if and only if all $N^{E_j}_0$ are compact (so that $\essspec{\square^{E_j}_0} \subseteq \{0\}$) and $\dim{\bergman(X_j,E_j)} = \dim{\ltwocohom[0,0]{X_j,E_j}} < \infty$ (so that $0 \not\in \essspec{\square^{E_j}_0}$ by item~\cref{item:N_bounded_zero_in_essential_spectrum} of \cref{inverse_of_laplacian_properties_bounded}).
 Because compactness of $N^{E_j}_0$ is equivalent to compactness of both $S^{E_j}_0 = 0$ and $S^{E_j}_1$ (see \cref{characterization_of_compactness_of_S}), \cref{item:dbar_neumann_noncompactness_riemann_surface_products_compactness_of_bot} follows.
 For \cref{item:dbar_neumann_noncompactness_riemann_surface_products_compactness_of_top}, we use the same argument with the formula
 $$ \essspec{\square^E_n} = \bigcup_{j=1}^n \bigg(\essspec{\square^{E_j}_1} + \sum_{j' \neq j} \spec{\square^{E_{j'}}_1}\bigg). $$
 Note that \cref{item:dbar_neumann_noncompactness_riemann_surface_products_compactness_of_bot,item:dbar_neumann_noncompactness_riemann_surface_products_compactness_of_top} are applications of the several factor version of item~\cref{item:characterization_of_compactness_of_S_for_product_complex_nondegenerate_compactness_of_factors} of \cref{characterization_of_compactness_of_S_for_product_complex}.
 If $q \in \{1,\dotsc,n-1\}$, then for every $1 \leq j \leq n$, there are $K \in \{0,1\}^n$ and $K'\in \{0,1\}^n$ which contribute to \cref{eq:dbar_neumann_noncompactness_riemann_surface_products_essential_spectrum}, and with $K_j = 0$ and $K'_j = 1$.
 Thus, $\essspec{\square^E_q} = \emptyset$ if and only if $\essspec{\square^{E_j}_0} = \essspec{\square^{E_j}_1} = \emptyset$ for all $1\leq j \leq n$, and this is equivalent to $N^E_0$ and $N^E_n$ being compact by the arguments in  \cref{item:dbar_neumann_noncompactness_riemann_surface_products_compactness_of_bot,item:dbar_neumann_noncompactness_riemann_surface_products_compactness_of_top}.
 This proves \cref{item:dbar_neumann_noncompactness_riemann_surface_products_compactness_in_middle}.
 
 Suppose $N^E_0$ is not compact.
 Then there must exist $j_0 \in \{1,\dotsc,n\}$ such that $\essspec{\square^{E_{j_0}}_0} \neq \emptyset$.
 Let $q \in \{0,\dotsc,n-1\}$ and pick $K \subseteq \{0,1\}^n$ with $\sum_{j=1}^n K_j = q$ and $K_{j_0} = 0$.
 Then $\essspec{\square^E_q}$ contains the infinite (since $\square^{E_{j'}}_0$ and $\square^{E_{j'}}_1$ are unbounded self-adjoint operators) set
 $$ \essspec{\square^{E_{j_0}}_0} + \sum_{j'\neq j_0} \spec{\square^{E_{j'}}_{K_{j'}}}, $$
 so $N^E_q$ is not compact.
 This proves \cref{item:dbar_neumann_noncompactness_riemann_surface_products_noncompactness_of_bot}, and a similar argument shows \cref{item:dbar_neumann_noncompactness_riemann_surface_products_noncompactness_of_top}.
 If there is $q_0 \in \{1,\dotsc,n-1\}$ such that $N^E_{q_0}$ is not compact,
 then one of $N^E_0$ and $N^E_n$ is not compact by \cref{item:dbar_neumann_noncompactness_riemann_surface_products_compactness_in_middle}, and \cref{item:dbar_neumann_noncompactness_riemann_surface_products_compactness_boils_up} follows by combining \cref{item:dbar_neumann_noncompactness_riemann_surface_products_noncompactness_of_bot,item:dbar_neumann_noncompactness_riemann_surface_products_noncompactness_of_top}.
 For \cref{item:dbar_neumann_noncompactness_riemann_surface_products_noncompactness_of_S}, combine \crefrange{item:dbar_neumann_noncompactness_riemann_surface_products_compactness_of_bot}{item:dbar_neumann_noncompactness_riemann_surface_products_compactness_in_middle}.
 
 If $j_0$ is as in \cref{item:dbar_neumann_noncompactness_riemann_surface_products_bergman_space_one_factor}, then $N^{E_{j_0}}_0$ fails to be compact by \cref{item:dbar_neumann_noncompactness_riemann_surface_products_compactness_of_bot}, and hence $N^E_q$ is not compact for $q \in \{0,\dotsc,n-1\}$ by \cref{item:dbar_neumann_noncompactness_riemann_surface_products_noncompactness_of_bot}.
\end{proof}

\appendix

\section{Joint spectra and the operator \texorpdfstring{$T\otimes \id{K} + \id{H} \otimes S$}{TxI + IxS}}\label{sec:joint_spectra}

By \cref{laplacian_of_product_hilbert_complex}, the spectrum of the Laplacian for the tensor product of two Hilbert complexes is determined by the closures of the operators $\Delta_j \otimes \id{H'_k} + \id{H_j} \otimes \Delta'_k$, with $\Delta$ and $\Delta'$ being the Laplacians for the individual factors.
Hence we are led to consider operators of the form $T \otimes \id{K} + \id{H}\otimes S$, where $T$ and $S$ are normal operators on Hilbert spaces $H$ and $K$, respectively.

We will make use of the Borel functional calculus for strongly commuting normal operators, where two normal operators on a common Hilbert space are said to \emph{strongly commute} if all their spectral projections commute.
If $T \coloneqq (T_1,\dotsc,T_n)$ is a tuple of pairwise strongly commuting normal operators, then the \emph{spectral theorem} (see \cite[Theorem~5.21]{Schmuedgen2012}) gives the existence of a joint spectral measure $P$ on the Borel sets of $\mathbb C^n$ such that
$$ T_k = \int_{\mathbb C^n} z_k \,dP(z_1,\dotsc,z_n). \quad (1 \leq k \leq n) $$
In fact, the spectral measure $P$ is the product of the spectral measures of $T_1, \dotsc, T_n$, so that
$$ P(M_1\times\dotsm\times M_n) = P_{T_1}(M_1)\dotsm P_{T_n}(M_n) $$
for Borel sets $M_k \subseteq \mathbb C$, $1\leq k \leq n$.

\begin{defn}\label{def:joint_spectrum}
 The \emph{joint spectrum} of the strongly commuting tuple $T = (T_1,\dotsc,T_n)$ is the support of $P$,
 $$ \spec{T} \coloneqq \big\{ z \in \mathbb C^n : P(B_\varepsilon(z)) \neq 0 \text{ for all } \varepsilon >0\big\}, $$
 where $B_\varepsilon(z)$ denotes the open ball in $\mathbb C^n$ with radius $\varepsilon$ and center $z$.
 The \emph{joint essential spectrum} of $T$ is
 $$ \essspec{T} \coloneqq \big\{z \in \mathbb C^n : \rank{P(B_\varepsilon(z))} = \infty \text{ for all } \varepsilon > 0\big\}. $$
 The complement of $\essspec{T}$ in $\spec{T}$ is called the \emph{joint discrete spectrum} of the tuple $T$,
 $$ \dspec{T} \coloneqq \big\{z \in \mathbb C^n : \exists \varepsilon_0 >0 \text{ such that } 0 < \rank{P(B_\varepsilon(z))} < \infty \text{ for all } \varepsilon \in (0,\varepsilon_0)\big\}. $$
\end{defn}

For $n=1$, these definitions reduce to the usual ones for a single operator.
The joint essential spectrum is closed in $\spec{T}$, and $\dspec{T}$ is discrete.
If $f \colon \spec{T} \to \mathbb C$ is a Borel measurable function, then we can use the joint spectral measure to define the normal operator
$$ f(T) \coloneqq \int_{\spec{T}} f\,dP. $$
The assignment $f \mapsto f(T)$ is called the \emph{Borel functional calculus for strongly commuting normal tuples}.
The spectrum of this operator is then the $P$-essential range of $f$,
$$ \spec{f(T)} = \big\{\lambda \in \mathbb C : P(f^{-1}(B_\varepsilon(\lambda))) \neq 0 \text{ for all } \varepsilon > 0\big\}, $$
and its essential spectrum is
$$ \essspec{f(T)} = \big\{\lambda \in \mathbb C : \rank{P(f^{-1}(B_\varepsilon(\lambda)))} = \infty \text{ for all } \varepsilon > 0\big\}. $$
Both of these formulas follow from the fact that the spectral measure associated to $f(T)$ is $P\circ f^{-1}$, where $f^{-1}$ is the preimage map on the Borel sets of $\mathbb C$.

\begin{thm}[Spectral mapping theorem]\label{spectral_mapping_theorem}
 Let $T = (T_1,\dotsc,T_n)$ be a tuple of pairwise strongly commuting normal operators and $f \colon \spec{T} \to \mathbb C$ continuous.
 Then
 $$ \spec{f(T)} = \overline{f(\spec{T})} \quad\text{and}\quad \essspec{f(T)} \supseteq \overline{f(\essspec{T})}. $$
 If $f$ is also \emph{proper} (meaning preimages of compact sets are compact), then
 $$ \spec{f(T)} = f(\spec{T}) \quad\text{and}\quad \essspec{f(T)} = f(\essspec{T}). $$
\end{thm}

\begin{proof}
 The spectral mapping theorem for the joint spectrum is well-known and can be found in \cite[Proposition~5.25]{Schmuedgen2012}.
 The proof of $\essspec{f(T)} \supseteq \overline{f(\essspec{T})}$ is similar to the corresponding inclusion for the joint spectrum:
 If $\lambda \in \overline{f(\essspec{T})}$ and $\varepsilon > 0$, then there is $z \in \essspec{T}$ with $|f(z) - \lambda| < \varepsilon/2$.
 Since $f$ is continuous, there exists $\delta > 0$ such that $B_\delta(z) \subseteq f^{-1}(B_\varepsilon(\lambda))$.
 Because $z$ is in the joint essential spectrum, $P(B_\delta(z))$ and hence also $P(f^{-1}(B_\varepsilon(\lambda)))$ has infinite rank, meaning $\lambda \in \essspec{f(T)}$.
 
 Now let $f \colon \spec{T} \to \mathbb C$ be proper.
 Then $f$ is a closed map, see \cite[Corollary]{Palais1970}.
 Hence we only have to show $\essspec{f(T)} \subseteq f(\essspec{T})$.
 If $\lambda \not\in f(\essspec{T})$, then there exists $\varepsilon > 0$ such that $f^{-1}(\overline{B_\varepsilon(\lambda)}) \cap \essspec{T} = \emptyset$.
 As $f$ is proper, the set $V\coloneqq f^{-1}(\overline{B_\varepsilon(\lambda)})$ is a compact subset of $\spec{T}$ contained in the joint discrete spectrum, implying that $P(V)$ and hence $P(f^{-1}(B_\varepsilon(\lambda)))$ has only finite dimensional range.
 Therefore, $\lambda \not\in \essspec{f(T)}$.
\end{proof}

Our principal example of a strongly commuting tuple will be the following:

\begin{lem}\label{joint_spectrum_of_operators_tensor_id}
 Let $T$ and $S$ be normal operators on Hilbert spaces $H$ and $K$, respectively.
 Then the operators $T\htensor \id{K}$ and $\id{H} \htensor S$ form a strongly commuting normal pair on the Hilbert space $H\htensor K$.
 We have
 $$ \spec{T\htensor \id{K},\id{H}\htensor S} = \spec{T}\times \spec{S} $$
 and
 $$ \essspec{T\htensor \id{K},\id{H}\htensor S} = \big(\essspec{T}\times \spec{S}\big)\cup \big(\spec{T} \times \essspec{S}\big). $$
\end{lem}

\begin{proof}
 The spectral measures of $T\htensor \id{K}$ and $\id{H} \htensor S$ are, respectively, given by
 $$ M \mapsto P_T(M)\htensor \id{K} \quad\text{and}\quad N \mapsto \id{H} \htensor P_S(N), $$
 where $P_T$ and $P_S$ are the spectral measures of $T$ and $S$, respectively.
 Therefore, the joint spectral measure of the pair $(T\htensor \id{K},\id{H}\htensor S)$ is given on rectangles $M\times N \subseteq \mathbb C^2$ by $P_T(M) \htensor P_S(N)$, and its image is
 $$ \img{P_T(M) \htensor P_S(N)} = \img{P_T(M)}\htensor \img{P_S(N)}. $$
 Now it follows that the image of $P_T(M)\htensor P_S(N)$ is nonzero (resp.\ infinite dimensional) if and only if both factors are nonzero (resp.\ at least one of them has infinite dimension and the other is nonzero).
 Since the products of open discs form a basis for the topology of $\mathbb C^2$, the result follows immediately.
\end{proof}

\begin{thm}\label{spectrum_of_T_plus_S}
 Let $T$ and $S$ be self-adjoint operators on Hilbert spaces $H$ and $K$, respectively.
 Put
 $$ A \coloneqq \int_{\mathbb R^2} (t + s)\,dP(t,s), $$
 where $P$ is the joint spectral measure of the pair $(T\htensor \id{K},\id{H}\htensor S)$.
 Then $A$ is the closure of the operator $T \otimes \id{K} + \id{H} \otimes S$ on $H\htensor K$.
 Moreover,
 \begin{equation}\label{eq:spectrum_of_T_plus_S}
  \spec{A} = \overline{\spec{T} + \spec{S}} \quad\text{and}\quad \essspec{A} \supseteq \overline{\essspec{T} + \spec{S}} \cup \overline{\spec{T} + \essspec{S}}.
 \end{equation}
 If, in addition, $T$ and $S$ are positive, then
 \begin{equation}\label{eq:spectrum_of_T_plus_S_positive}
  \spec{A} = \spec{T} + \spec{S} \quad\text{and}\quad \essspec{A} = \big(\essspec{T} + \spec{S}\big) \cup \big(\spec{T} + \essspec{S}\big).
 \end{equation}
\end{thm}

\begin{proof}
 It is easy to show that $A$ is a self-adjoint extension of $T \otimes \id{K} + \id{H} \otimes S$.
 Because $T\otimes \id{K} + \id{H} \otimes S$ is essentially self-adjoint, see \cite[Theorem~VIII.33]{Reed1980}, $A$ must be its closure. 
 Equation~\cref{eq:spectrum_of_T_plus_S} follows from \cref{joint_spectrum_of_operators_tensor_id,spectral_mapping_theorem} by applying it to the function $f \colon \mathbb C^2 \to \mathbb C$, $f(t,s) \coloneqq t+s$.
 If $T$ and $S$ are positive, then so are $T\htensor \id{K}$ and $\id{H} \htensor S$, and hence $\spec{T\htensor \id{K},\id{H}\htensor S}$ is contained in $[0,\infty)\times [0,\infty)$.
 On this set, $f$ is proper and \cref{eq:spectrum_of_T_plus_S_positive} follows, again, from \cref{joint_spectrum_of_operators_tensor_id,spectral_mapping_theorem}.
\end{proof}

\begin{rem}
 Of course, the joint spectrum of $(T\htensor \id{K}, \id{H}\htensor S)$ and the spectrum of the closure of $T\otimes \id{K} + \id{H}\otimes S$ are well-known in the literature, see for instance \cite[Lemma~7.24]{Schmuedgen2012} or \cite[Theorem~VIII.33]{Reed1980}.
 However, the corresponding statements regarding their \emph{essential} spectrum, as well as the spectral mapping theorem for $\essspec{f(T)}$, seem to be new (at least to the knowledge of the author).
\end{rem}

\printbibliography

\end{document}

%% file: macros.tex
\newcommand{\id}[1]{{\mathrm{id}_{#1}}} 
\newcommand{\img}[1]{{\mathrm{img}(#1)}} 
\renewcommand{\dim}[2][]{{\mathrm{dim}_{#1}(#2)}} 
\newcommand{\rank}[1]{{\mathrm{dim}\,\img{#1}}} 

\newcommand{\spec}[2][]{{\sigma_{#1}(#2)}} 
\newcommand{\essspec}[1]{{\sigma_e(#1)}} 
\newcommand{\dspec}[1]{{\sigma_d(#1)}} 
\newcommand{\dom}[1]{{\mathrm{dom}(#1)}} 
\newcommand{\hol}[1]{{\mathcal O(#1)}} 

\newcommand{\bergman}[1][2]{A^{#1}} 
\newcommand{\dbar}{{\smash{\overline{\partial}}}} 
\newcommand{\supp}{{\mathrm{supp}}} 

\newcommand{\htensor}{\mathbin{\hat\otimes}} 
\newcommand{\hgtensor}{\mathbin{\hat\circledast}} 
\newcommand{\gtensor}{\circledast} 
\newcommand{\getensor}{\circledast}

\newcommand{\minn}{s} 
\newcommand{\maxx}{w} 

\newcommand{\cohom}[2][]{{\mathcal H^{#1}(#2)}} 
\newcommand{\redcohom}[2][]{{\smash{\overline{\mathcal H}}^{#1}(#2)}} 
\newcommand{\ltwocohom}[2][]{{\mathcal H^{#1}_{L^2}(#2)}} 
\newcommand{\redltwocohom}[2][]{{\smash{\overline{\mathcal H}}^{#1}_{L^2}(#2)}} 

\newlength{\hilbcomparrow}
\newcommand*{\hilbcomprightarrow}[1]{\xrightarrow{\mathmakebox[\hilbcomparrow]{#1}}}
\newcommand*{\hilbcompleftarrow}[1]{\xleftarrow{\mathmakebox[\hilbcomparrow]{#1}}}

\renewcommand{\implies}{\,\Rightarrow\,}

\renewcommand{\iff}{\,\Leftrightarrow\,}